\documentclass[12pt]{amsart}
\usepackage{a4wide,enumerate,xcolor,hyperref}
\usepackage{amsmath}
\allowdisplaybreaks

\let\pa\partial
\let\na\nabla
\let\eps\varepsilon
\newcommand{\N}{{\mathbb N}}
\newcommand{\R}{{\mathbb R}}
\newcommand{\diver}{\operatorname{div}}

\newcommand{\F}{\mathcal{F}}
\newcommand{\m}{\mathrm{m}}

\newtheorem{theorem}{Theorem}
\newtheorem{lemma}[theorem]{Lemma}

\newtheorem{remark}[theorem]{Remark}
\newtheorem{corollary}[theorem]{Corollary}

\begin{document}

\title[Charge transport systems for memristors]{Charge transport systems with \\ Fermi--Dirac statistics for memristors}

\author[M. Herda]{Maxime Herda}
\address{Univ. Lille, Inria, CNRS, UMR 8524 - Laboratoire Paul Painlevé, F-59000 Lille, France}
\email{maxime.herda@inria.fr}

\author[A. J\"ungel]{Ansgar J\"ungel}
\address{Institute of Analysis and Scientific Computing, TU Wien, Wiedner Hauptstra\ss e 8--10, 1040 Wien, Austria}
\email{juengel@tuwien.ac.at} 

\author[S. Portisch]{Stefan Portisch}
\address{Institute of Analysis and Scientific Computing, TU Wien, Wiedner Hauptstra\ss e 8--10, 1040 Wien, Austria}
\email{stefan.portisch@tuwien.ac.at} 

\date{\today}

\thanks{The first author acknowledges support from the LabEx CEMPI (ANR-11-LABX-0007). The last two authors acknowledge partial support from  the Austrian Science Fund (FWF), grant DOI 10.55776/P33010 and 10.55776/F65. This work has received funding from the European Research Council (ERC) under the European Union's Horizon 2020 research and innovation programme, ERC Advanced Grant NEUROMORPH, no.~101018153.} 

\begin{abstract}
An instationary drift--diffusion system for the electron, hole, and oxygen vacancy densities, coupled to the Poisson equation for the electric potential, is analyzed in a bounded domain with mixed Dirichlet--Neumann boundary conditions. The electron and hole densities are governed by Fermi--Dirac statistics, while the oxygen vacancy density is governed by Blakemore statistics. The equations model the charge carrier dynamics in memristive devices used in semiconductor technology. The global existence of weak solutions is proved in up to three space dimensions. The proof is based on the free energy inequality, an iteration argument to improve the integrability of the densities, and estimations of the Fermi--Dirac integral. Under a physically realistic elliptic regularity condition, it is proved that the densities are bounded.
\end{abstract}

 

\maketitle


\section{Introduction}

Memristors are nonlinear resistors with memory able to exhibit a resistive switching behavior. In neuromorphic computing, they are used to build artificial neurons and synapses \cite{IeAm20}. Also perovskite solar cells may show a memristive behavior, emulating synaptic- and neural-like dynamics \cite{TeVa20}. In semiconductor technology, often oxide-based memristors are used. They consist of a thin titanium dioxide layer between two metal electrodes \cite{Mla19}. Charge carriers are the electrons, holes (defect electrons), and oxide vacancies which allow for a modulation of the layer conductance. 

Generally, the relation between the electron density and its chemical potential (quasi-Fermi potential) is given by Fermi--Dirac statistics. In low-density regimes, this reduces to Maxwell--Boltzmann statistics, leading to particle fluxes with linear diffusion \cite{Jue09}, while in  high-density regimes, Fermi--Dirac statistics reduce to a power-law density--chemical potential relation, leading to fluxes with degenerate diffusion. A mathematical analysis of the associated low-density drift--diffusion equations was performed in \cite{JJZ23}, while high-density models were studied in \cite{JuVe23}. In this paper, we investigate for the first time a general drift-diffusion system with Fermi--Dirac statistics for the electrons and holes as well as physically motivated Blakemore statistics for the oxide vacancies.

\subsection{Model equations}

The charge transport through the semiconductor device is supposed to be governed by the mass balance equations for the electron density $n(x,t)$, hole density $p(x,t)$, and density $D(x,t)$ of oxide vacancies, and the gradients of the associated chemical potentials (quasi-Fermi potentials) $\mu_n$, $\mu_p$, and $\mu_D$ are the driving forces of the flow. This leads to the (scaled) equations
\begin{align*}
  \pa_t n - \diver J_n = 0, &\quad J_n = n\na\mu_n, \\  
  \pa_t p + \diver J_p = 0, &\quad J_p = -p\na\mu_p, \\
  \pa_t D + \diver J_D = 0, &\quad J_D = -D\na\mu_D,
\end{align*}
where $J_n$, $J_p$, and $J_D$ are the electron, hole, and oxide vacancy current densities, respectively. Fermi--Dirac statistics is valid for electrons in the conduction band and for holes in the valance band in the parabolic band approximation \cite[Sec.~1.6]{Jue09}, giving the relations
\begin{align*}
  n = \F_{1/2}(\mu_n+V), \quad p = \F_{1/2}(\mu_p-V),
\end{align*}
where $V$ denotes the electric potential, and the Fermi--Dirac integral is defined by
\begin{align*}
  \F_{1/2}(y) = \frac{2}{\sqrt{\pi}}\int_0^\infty
  \frac{\sqrt{s}}{1+e^{s-y}}ds, \quad y\in\R.
\end{align*}
In the Maxwell--Boltzmann approximation, the Fermi--Dirac integral can be approximated by the exponential, $\F_{1/2}(y)\approx\exp(y)$ for $y\ll -1$, leading to the electron flux $J_n\approx n\na(\log n-V)=\na n-n\na V$. However, the use of Fermi--Dirac statistics is more appropriate in regimes with moderate or high densities. We expect that the oxide vacancies cannot be accumulated excessively such that it is reasonable to use Blakemore statistics \cite{Bla82},
\begin{align*}
  D = \F_{-1}(\mu_D-V),\quad\mbox{where }\F_{-1}(y) = \frac{1}{1+e^{-y}},
  \ y\in\R.
\end{align*}
Although being itself an approximation of Fermi--Dirac statistics, Blackmore statistics have the advantage of restricting the oxide vacancy density to the interval $(0,1)$.  Without loss of generality, we have set the upper bound equal to one.

Introducing the inverse functions
\begin{align*}
  g(z) &= \F_{1/2}^{-1}(z)\quad\mbox{for }z\in(0,\infty), \\
  h(z) &= \F_{-1}^{-1}(z) = \log z - \log(1-z)\quad\mbox{for }z\in(0,1),
\end{align*}
the transport equations can be written in a drift--diffusion form as
\begin{align}
  \pa_t n - \diver J_n = 0, 
  &\quad J_n = n\na g(n) - n\na V, \label{1.n} \\
  \pa_t p + \diver J_p = 0, 
  &\quad J_p = -(p\na g(p)+p\na V), \label{1.p} \\
  \pa_t D + \diver J_D = 0, 
  &\quad J_D = -(D\na h(D)+D\na V)
  \quad\mbox{in }\Omega,\ t>0, \label{1.D}
\end{align}
where $\Omega\subset \R^d$ ($d\ge 1$) is a bounded domain. The electric potential is selfconsistently coupled to the charge densities by the Poisson equation
\begin{align}\label{1.V}
  \lambda^2\Delta V = n-p-D+A(x)\quad\mbox{in }\Omega,
\end{align}
where $\lambda>0$ is the (scaled) Debye length and $A(x)$ is the given dopant acceptor density. Following \cite{SBW09}, we neglect recombination--generation effects. Equations \eqref{1.n}--\eqref{1.V} are supplemented with the initial and mixed Dirichlet--Neumann boundary conditions
\begin{align}
  n(0,\cdot) = n^I, \quad p(0,\cdot) = p^I, \quad D(0,\cdot) = D^I
  &\quad\mbox{in }\Omega, \label{1.ic} \\
  n = \bar{n}, \quad p=\bar{p}, \quad V=\bar{V} &\quad\mbox{on }
  \Gamma_D,\ t>0, \label{1.Dbc} \\
  J_n\cdot\nu = J_p\cdot\nu = \na V\cdot\nu = 0 &\quad\mbox{on }
  \Gamma_N,\ t>0, \label{1.noflux1} \\
  J_D\cdot\nu = 0 &\quad\mbox{on }\pa\Omega, t>0. \label{1.noflux2}
\end{align}
Here, $\Gamma_D$ is the union of Ohmic contacts and $\Gamma_N$ models the insulating boundary parts. Since the oxide vacancies are supposed not to leave the domain, we impose no-flux boundary conditions for $D$ on the whole boundary. These boundary conditions are usually used in the literature \cite{GSTD13,SBW09}. 

The aim of this paper is to prove (i) the existence of global weak solutions $(n,p,D,V)$ to \eqref{1.n}--\eqref{1.noflux2} and (ii) the regularity $n,p,D\in L^\infty(0,T;L^\infty(\Omega))$ for any $T>0$. 

\subsection{Mathematical difficulties}

The misfit of the boundary conditions for $(n,p)$ on the one hand and for $D$ on the other hand gives the first main mathematical difficulty. A second difficulty comes from the fact that we consider three species instead of two charge carriers as done in many papers \cite{GaGr89,GlLi19,Jue96}. Indeed, the two-species case allows one to exploit a monotonicity property of the drift term such that the quadratic nonlinearity can be handled \cite{GaGr89}. For more than two species, one may use Gagliardo--Nirenberg estimates, but this is possible in two space dimensions only \cite{GlHu05}. This issue can be overcome by $W^{1,r}_{\rm loc}(\Omega)$ estimates with $r>1$ \cite{JJZ23}, but leading to very weak solutions and boundedness of solutions in two space dimensions only. The third difficulty are the nonlinearities from the Fermi--Dirac statistics, which complicates the estimates. We prove in Appendix \ref{sec.app} that 
\begin{align}\label{1.Dg}
  g'(z) \sim z^{-1}\mathrm{1}_{\{z\le\F_{1/2}(0)\}}
  + z^{-1/3}\mathrm{1}_{\{z>\F_{1/2}(0)\}}\quad\mbox{for }z>0,
\end{align} 
where $A\sim B$ means that there exist constants $C_1,C_2>0$ such that $C_1A\le B\le C_2A$. In particular, the nonlinear diffusion $n\na g(n)=ng'(n)\na n$ can be approximated by $\na n$ in the low-density regime and by $(3/5)\na n^{5/3}$ in the high-density regime. On the other hand, the Blakemore statistics gives to the diffusion $D\na h(D)=-\na\log(1-D)$, which exhibits a singularity at $D=1$. The technical issues associated to this singularity are overcome by using some ideas from \cite{CCFG21}, developed for a one-species model.

\subsection{State of the art and key ideas}

There are only a few works dealing with the drift--diffusion equations for more than two species. General existence results for an $n$-species model have been proved in \cite{HPR19} for an abstract drift operator satisfying smoothing conditions. In \cite{ChLu95,GlHu97,GlHu05,GlLi19}, the existence of global weak solutions was shown in at most two space dimensions. The three-dimensional case was investigated in \cite{BFPR14} using Robin boundary conditions for the electric potential. In the work \cite{BGN22}, the function $n\na g(n)=\na(n+\eta n^q)$ with $\eta>0$ and $q\ge 4$ was chosen to regularize the diffusion term, which allows for an analysis in three space dimensions. The paper \cite{GaGr96} studies the drift--diffusion equations with Fermi--Dirac statistics but assuming inhomogeneous Neumann boundary conditions on $\pa\Omega$. A drift--diffusion system with Fermi--Dirac statistics for electrons and holes and with Blakemore statistics for the ionic vacancy carriers, modeling perovskite solar cells, was analyzed recently in \cite{AGL24} in two space dimensions. A free energy inequality for this model in three space dimensions was shown in \cite{ACFH23}. 

Our analysis is based, as in \cite{ACFH23,GaGr96}, on estimates derived from the free energy inequality. The asymptotic behavior of the Fermi--Dirac integral $\F_{1/2}$ allows for an argument similar to \cite{BGN22} but based on physical bounds. Indeed, the behavior \eqref{1.Dg} shows that the diffusion is given by
\begin{align*}
  n\na g(n) \sim n(n^{-1}+n^{-1/3})\na n 
  = \na\bigg(n + \frac35 n^{5/3}\bigg).
\end{align*}
The first term corresponds to linear diffusion, while the second term allows for higher integrability estimates. As a by-product, we are able to weaken the condition $q\ge 4$ in \cite{BGN22} to $q\ge 5/3$. (By \cite{JuVe23}, one may weaken this condition even to $q>6/5$.)

To specify the free energy inequality, we introduce the anti-derivatives of $g$ and $h$,
\begin{align}\label{1.GH}
  G(s) = \int_{\F_{1/2}(0)}^s g(z)dz, \quad
  H(s) = \int_{\F_{-1}(0)}^s h(z)dz,
\end{align}
the relative energy density
\begin{align*}
  \mathcal{G}(s|\bar{s}) = G(s) - G(\bar{s}) - G'(\bar{s})(s-\bar{s}),
  \quad s,\bar{s}\ge 0,
\end{align*}
and the free energy
\begin{align*}
  E(n,p,D,V) = \int_\Omega\bigg(\mathcal{G}(n|\bar{n}) 
  + \mathcal{G}(p|\bar{p})
  + H(D) + D\bar{V} + \frac{\lambda^2}{2}|\na(V-\bar{V})|^2\bigg)dx.
\end{align*} 
A formal computation, made rigorous in Theorem \ref{thm.ex}, shows that
\begin{align}\label{1.dEdt}
  \frac{dE}{dt}(n,p,D,V) &+ \frac12\int_\Omega\big(
  n|\na(g(n)-V)|^2 + p|\na(g(p)+V)|^2 \\
  &+ D|\na(h(D)+V)|^2\big)dx
  \le C(\bar{n},\bar{p},\bar{D},T). \nonumber
\end{align}
This yields a priori estimates for $n$, $p$ in $L^\infty(0,T;L^{5/3}(\Omega))$ and for $V$ in $L^\infty(0,T;H^1(\Omega))$. Moreover, defining $\widetilde{g}$ by $\widetilde{g}'(n)=\sqrt{n}g'(n)$, 
\begin{align*}
  |\na\widetilde{g}(n)| \le |\na\widetilde{g}(n)-\sqrt{n}\na V| 
  + \sqrt{n}|\na V|
  = \sqrt{n}|\na(g(n)-V)| + \sqrt{n}|\na V|
\end{align*}
is uniformly bounded in $L^2(0,T;L^{5/4}(\Omega))$. Unfortunately, this regularity is {\em not} sufficient to define
$n\na g(n) = \sqrt{n}\na\widetilde{g}(n)$ since $\sqrt{n}$ is bounded in $L^\infty(0,T;L^{10/3}(\Omega))$ and $3/10+4/5>1$. However, we are able to improve the regularity by an iteration argument to $\na\widetilde{g}(n)\in L^2(0,T;L^r(\Omega))$ with $r<8/5$ (see Lemma \ref{lem.est2}), which is sufficient since $3/10+5/8<1$. 

The treatment of the diffusion $D\na h(D)=-\na\log(1-D)$ is quite delicate because of the singularity at $D=1$. The idea is to approximate $L(D)=-\log(1-D)$ by regular functions $L_k$ with $k\in\N$. The identification of the limit of the sequence $L_k(D_k)$ of approximating solutions $D_k$ which converge strongly to some function $D$ is then achieved by a monotonicity argument (Minty trick); see Lemma \ref{lem.Lstar}. These ideas allow us to prove the existence of global weak solutions.

The second main result is the boundedness of weak solutions. The difficulty comes from the estimate of the quadratic drift terms, which can be overcome in the case of two species by a monotonicity argument. For more than two species, we use the Gagliardo--Nirenberg inequality to estimate this term, similarly as in \cite{GlHu97} for two space dimensions. In three dimensions, we need as in \cite{JuVe23} the elliptic regularity result $V\in W^{1,r}(\Omega)$ with $r>3$. This is possible even under mixed boundary conditions if $\Gamma_D$ and $\Gamma_N$ do not meet in a ``too wild'' manner \cite[Theorem 4.8]{DiRe15}. Then, applying an Alikakos-type iteration argument similar to \cite{JJZ23,JuVe23}, we obtain $q$-uniform estimates in $L^\infty(0,T;L^{q}(\Omega))$ for any $q<\infty$. The boundedness follows after performing the limit $q\to\infty$.

\subsection{Main results} 

First, we introduce some notation. We set $\Omega_T=\Omega\times(0,T)$ for $T>0$, denote by $\m(B)$ the measure of a set $B\subset\R^d$, and set for $1\le q\le\infty$,
\begin{align*}
  W_D^{1,q}(\Omega) = \{u\in W^{1,q}(\Omega): u=0\mbox{ on }\Gamma_D\},
  \quad H_D^1(\Omega) = W_D^{1,2}(\Omega).
\end{align*}
The function $V^I\in H_D^1(\Omega)+\bar{V}$ is the unique solution to
\begin{equation*}
  \lambda^2\Delta V^I = n^I-p^I-D^I+A(x)\mbox{ in }\Omega, \quad
  V^I=\bar{V}\mbox{ on }\Gamma_D, \quad \na V^I\cdot\nu = 0
  \mbox{ on }\Gamma_N.
\end{equation*}
Constants $C>0$ are generic and may change their value from line to line.

We impose the following assumptions.

\begin{itemize}
\item[(A1)] Domain: $\Omega \subset \R^d$ $(1\le d\le 3)$ is a bounded domain with Lipschitz boundary, $\partial \Omega = \Gamma_D \cup \Gamma_N$,  $\Gamma_D \cap \Gamma_N = \emptyset$, $\m(\Gamma_D) > 0$, and $\Gamma_N$ is relatively open in $\partial \Omega$.
\item[(A2)] Data: $T > 0$, $\lambda > 0$, $A \in L^\infty(\Omega)$.
\item[(A3)] Boundary data: $\bar{n}$, $\bar{p}$, $\bar{V} \in W^{1,\infty}(\Omega)$ with $\bar{n}$, $\bar{p} > 0$ in $\Omega$.
\item[(A4)] Initial data: $n^I$, $p^I$, $D^I \in L^2(\Omega)$ satisfy $n^I$, $p^I$, $D^I \geq 0$ in $\Omega$, $E(n^I, p^I, D^I, V^I) < \infty$. Furthermore, $\sup_\Omega D^I\le 1$ and 
\begin{align*}
  D_\Omega^I:=\frac{1}{\m(\Omega)}\int_\Omega D^I dx < 1.  
\end{align*} 
\item[(A5)] Elliptic Regularity: There exists $r>3$ such that for some constant $C > 0$ and all $f \in L^{3r/(r+3)}(\Omega)$ the weak solution $V$ to the Poisson problem
\begin{equation}\label{1.ellip}
  \Delta V = f\mbox{ in } \Omega, \quad V = \bar{V}\mbox{ on } \Gamma_D, \quad \nabla V \cdot \nu = 0\mbox{ on } \Gamma_N,
\end{equation}
satisfies the estimate
\begin{equation}\label{1.W1r}
  \| V \|_{W^{1,r}(\Omega)} \leq C \|f\|_{L^{3r/(r+3)}(\Omega)} + C.
\end{equation}
\end{itemize}

Let us discuss the assumptions. We can assume higher space dimensions in most of the estimates, but we restrict ourselves to $d\le 3$ because of the applications. The boundary data in Assumption (A3) is assumed to be independent of time to simplify the computations; time-dependent boundary data are possible, see, e.g., \cite[Sec.~2]{DGJ97}. Compared to \cite{AGL24}, we do not need pointwise positive lower bounds of the densities and we can allow for vacuum as well as saturation of the oxygen vacancy density. We only prevent  $D_\Omega^I=1$ in Assumption (A4), which would be physically unrealistic. 

The most restrictive condition is Assumption (A5). Indeed, we can only expect the regularity $V\in W^{1,r}(\Omega)$ with $r>2$ for the solution $V$ to \eqref{1.ellip} with mixed boundary conditions \cite{Gro94}. Shamir's counterexample \cite{Sha68} shows that $r<4$ is generally necessary, even for smooth domains and data. The regularity $r>3$ can be achieved under reasonable conditions on $\Gamma_D$ and $\Gamma_N$ \cite[Theorem 4.8]{DiRe15}. These conditions are satisfied if $\Gamma_D$ and $\Gamma_N$ intersect with an ``angle'' not larger than $\pi$ \cite[Prop.~3.4]{DiRe15}. Assumption (A5) is {\em not} needed for the existence result but for the proof of the boundedness of solutions.

Our first main result is the existence of global weak solutions.

\begin{theorem}[Global existence]\label{thm.ex}
Let Assumptions (A1)--(A4) hold. Then there exists a weak solution $(n,p,D,V)$ to \eqref{1.n}--\eqref{1.noflux2} satisfying $n,p\ge 0$, $0\le D<1$ a.e.\ in $\Omega_T$,
\begin{align*}
  & n,p\in L^\infty(0,T;L^{5/3}(\Omega))
  \cap L^2(0,T;W^{1,\alpha}(\Omega)),
  \quad D\in L^\infty(\Omega_T)\cap L^2(0,T;H^1(\Omega)), \\
  & n\na g(n),p\na g(p)\in L^2(0,T;L^{5/4}(\Omega)), \quad
  D\na h(D)\in L^2(0,T;L^{2}(\Omega)), \\
  & \pa_t n,\pa_t p\in L^{7/5}(0,T;W_D^{1,2\alpha/(4-\alpha)}(\Omega)), \quad
  \pa_t D\in L^2(0,T;H^1(\Omega)'), \quad V\in L^\infty(0,T;H^1(\Omega)),
\end{align*}
where $\alpha<8/5$ if $d=3$, $\alpha<2$ if $d=2$, and $\alpha=2$ if $d=1$. The fluxes are understood in the sense
\begin{align*}
  J_n &= n\na(g(n)-V)\in L^2(0,T;L^{5/4}(\Omega)), \\
  J_p &= -p\na(g(p)+pV)\in L^2(0,T;L^{5/4}(\Omega)), \\
  J_D &= -D\na(h(D)+V) = -\na\log(1-D)+D\na V\in L^2(\Omega_T).
\end{align*}
The solution satisfies the free energy inequality
\begin{align}\label{1.Eineq}
  E(n,p,D,V)(t) + \frac12\int_0^t\int_\Omega\big(&
  n|\na(g(n)-V)|^2 + p|\na(g(p)+V)|^2 \\
  &+ D|\na(h(D)+V)|^2\big)dxds
  \le C(E^I,\Lambda,T), \nonumber
\end{align}
where $E^I:=E(n^I,p^I,D^I,V^I)$, 
\begin{align*}
  \Lambda := 2\big(\|\na(g(\bar{n})-\bar{V})\|_{L^\infty(\Omega)}^2
  + \|\na(g(\bar{p})+\bar{V})\|_{L^\infty(\Omega)}^2\big),
\end{align*}
and it holds that $C(E^I,\Lambda,T)=0$ if $\Lambda=0$.
\end{theorem}

The property $\Lambda=0$ means that the boundary data is in thermal equilibrium. In this situation, the free energy is a Lyapunov functional. For the proof of Theorem \ref{thm.ex}, we first approximate the problem by truncating the nonlinearities (densities) in the diffusion and drift terms and prove the existence of approximate solutions by using the Leray--Schauder fixed-point theorem. The compactness of the fixed-point operator is a consequence of the approximate free energy inequality. From this inequality, we derive uniform bounds for the approximate solutions, allowing us to take the re-regularizing limit. As mentioned before, the main difficulties are the derivation of improved estimates via an iteration argument and the treatment of the singularity $D=1$. 

Our second main result is the boundedness of weak solutions.

\begin{theorem}[Boundedness]\label{thm.bound}
Let Assumptions (A1)--(A5) hold and assume that $n^I$, $p^I$, $D^I\in L^\infty(\Omega)$. Then the weak solution constructed in Theorem \ref{thm.ex} satisfies
\begin{align*}
  n,p,D\in L^\infty(\Omega_T), \quad V\in L^{\infty}(0,T;W^{1,r}(\Omega))
  \subset L^\infty(\Omega_T),
\end{align*}
where $r>3$ is given in Assumption (A5). 
\end{theorem}

The restriction to three space dimensions comes from regularity \eqref{1.W1r}. The boundedness result is not surprising in view of \cite[Theorem 2]{JuVe23}. Indeed, since $ng'(n)\sim 1+n^{2/3}$, the diffusion term contains the porous-medium term $\na n^{5/3}$, and it is proved in \cite{JuVe23} that this nonlinear diffusion leads to an improvement of the integratibility of the densities up to $L^\infty(\Omega)$. The idea is first to prove that $n,p\in L^\infty(0,T;L^2(\Omega))$. This is used as the starting point of a recursion showing that $n,p\in L^\infty(0,T;L^q(\Omega))$ for any $q<\infty$, but with bounds that may depend on $q$. This allows us to use $n^q-\bar{n}^q$, $p^q-\bar{p}^q$ as test functions in the weak formulations to \eqref{1.n}, \eqref{1.p}, respectively. By an Alikakos iteration, it turns out that the $L^\infty(0,T;L^q(\Omega))$ bounds are independent of $q$, and we can pass to the limit $q\to\infty$ to conclude. To reduce the technicalities and since the first parts of the proof are technically similar to \cite[Sec.~3]{JuVe23}, we detail only the last part of the proof (the Alikakos argument). 

\begin{remark}[Generalization]\rm
Our results hold for an arbitrary number of charged particles, since we use the Poisson equation only through the norm estimates for $V$ and $\na V$. In particular, we can consider the transport equations
\begin{align*}
  \pa_t u_i &= \diver(u_i\na g(u_i) + z_iu_i\na V), \quad i\in I, \\
  \pa_t u_i &= \diver(u_i\na h(u_i) + z_iu_i\na V), \quad i\in I_0, \\
  \lambda^2\Delta V &= -\sum_{i\in I\cup I_0}z_iu_i + A(x)\quad
  \mbox{in }\Omega,\ t>0,
\end{align*}
where $z_i\in\R$ are the particle charges, $I,I_0\subset\N$ are some index sets, and the initial and boundary conditions are as in \eqref{1.ic}--\eqref{1.noflux2}.
\qed\end{remark}

The paper is organized as follows. Theorem \ref{thm.ex} and \ref{thm.bound} are proved in Sections \ref{sec.ex} and \ref{sec.bound}, respectively. Auxiliary inequalities involving Fermi--Dirac integrals  are proved in Appendix \ref{sec.app}. We also need a nonlinear version of the Poincar\'e--Wirtinger inequality, which is shown in Appendix \ref{sec.poincare}.


\section{Proof of Theorem \ref{thm.ex}}\label{sec.ex}

We prove the existence of global weak solutions to \eqref{1.n}--\eqref{1.noflux2}. To this end, we truncate the coefficients in the parabolic equations with parameter $k\in\N$, solve the corresponding approximate problem, derive uniform estimates from an approximate free energy inequality, and pass to the limit $k\to\infty$. 

\subsection{Approximate problem}

We introduce for $k\in\N$ and $z\in\R$ the truncations $T_k(z)=\max\{0,\min\{k,z\}\}$ and
\begin{align*}
  S_k^1(z) = \begin{cases}
  1 &\mbox{for }z\le 0, \\
  zg'(z) &\mbox{for }0<z\le k, \\
  k^{2/3}z^{1/3}g'(z) &\mbox{for }z>k,  
  \end{cases} \quad
  S_k^2(z) = \begin{cases}
  1 &\mbox{for }z\le 0, \\
  zh'(z) &\mbox{for }0<z\le k/(k+1), \\
  1+k &\mbox{for }z>k/(k+1).
  \end{cases}
\end{align*}
The functions $S_k^1$ and $S_k^2$ are continuous, bounded, and strictly positive on $\R$ noting that $zh'(z)=1/(1-z)$ for $z\in(0,1)$. The approximate problem reads as follows:
\begin{align}
  \pa_t n_k &= \diver\big(S_k^1(n_k)\na n_k - T_k(n_k)\na V_k\big), 
  \label{2.nk} \\
  \pa_t p_k &= \diver\big(S_k^1(p_k)\na p_k + T_k(p_k)\na V_k\big), 
  \label{2.pk} \\
  \pa_t D_k &= \diver\big(S_k^2(D_k)\na D_k + T_{k/(k+1)}(D_k)\na V_k
  \big), \label{2.Dk} \\
  \lambda^2\Delta V_k &= n_k-p_k-D_k+A(x)\quad\mbox{in }\Omega,\ t>0, 
  \label{2.Vk}
\end{align}
with the initial and boundary conditions \eqref{1.ic}--\eqref{1.noflux2}, where $(n,p,D,V)$ is replaced by $(n_k,p_k,$ $D_k,V_k)$. Clearly, if $k\to\infty$, we recover formulation \eqref{1.n}--\eqref{1.D}. The truncation $T_{k/(k+1)}(D_k)$ is chosen since we expect that the limit $D$ of $D_k$ satisfies $D<1$ a.e.

We show the existence of solutions to \eqref{2.nk}--\eqref{2.Vk} by using a fixed-point argument. For this, let $(n^*,p^*,D^*)\in L^2(\Omega_T)^3$ and $\sigma\in[0,1]$. We apply \cite[Theorem 23.A]{Zei90} to infer that the linearized problem
\begin{align}
  \pa_t n &= \diver\big(S_k^1(n^*)\na n - \sigma T_k(n^*)\na V\big), 
  \label{2.linn} \\
  \pa_t p &= \diver\big(S_k^1(p^*)\na p + \sigma T_k(p^*)\na V\big), 
  \label{2.linp} \\
  \pa_t D &= \diver\big(S_k^2(D^*)\na D + \sigma T_{k/(k+1)}(D^*)\na V
  \big), \label{2.linD} \\
  \lambda^2\Delta V &= n^*-p^*-D^*+\sigma A(x)
  \quad\mbox{in }\Omega,\ t>0, \label{2.linV}
\end{align}
with the initial and boundary conditions
\begin{align*}
  n(0,\cdot) = \sigma n^I, \quad p(0,\cdot) = \sigma p^I, 
  \quad D(0,\cdot) = \sigma D^I
  &\quad\mbox{in }\Omega, \\
  n = \sigma\bar{n}, \quad p=\sigma\bar{p}, \quad 
  V=\sigma\bar{V} &\quad\mbox{on }
  \Gamma_D,\ t>0, \\
  \na n\cdot\nu = \na p\cdot\nu = \na V\cdot\nu = 0 &\quad\mbox{on }
  \Gamma_N,\ t>0, \\
  \big(S_k^2(D^*)\na D + \sigma T_{k/(k+1)}(D^*)\na V
  \big)\cdot\nu = 0 &\quad\mbox{on }\pa\Omega, t>0,
\end{align*}
has a unique solution $(n,p,D,V)\in L^2(0,T;H^1(\Omega))^4$ such that $n,p,D\in H^1(0,T;H^1_D(\Omega)')$. This defines the fixed-point operator $F:L^2(\Omega)^3\times[0,1]\to L^2(\Omega_T)^3$, $(n^*,p^*,D^*;\sigma)\mapsto (n,p,D)$. Standard arguments show that $F$ is continuous and satisfies $F(n^*,p^*,D^*;0)=(0,0,0)$. To apply the Leray--Schauder fixed-point theorem, we need to find a uniform bound for all fixed points of $F(\cdot,\cdot,\cdot;\sigma)$. 

\begin{lemma}\label{lem.FP}
Let $(n,p,D)$ be a fixed point of $F(\cdot,\cdot,\cdot;\sigma)$, where $\sigma\in[0,1]$. Then $(n,p,D)$ is bounded in $L^\infty(0,T;L^2(\Omega))\cap L^2(0,T;H^1(\Omega))$ uniformly in $\sigma$.
\end{lemma}

\begin{proof}
Since the proof is similar to that one of \cite[Lemma 2.1]{JJZ23}, we only sketch it. Let $(n^*,p^*,D^*)=(n,p,D)$ be a fixed point of $F(\cdot,\cdot,\cdot;\sigma)$. We use the test function $V-\sigma\bar{V}$ in the weak formulation of \eqref{2.linV} and apply the Young and Poincar\'e inequality to find that 
\begin{align*}
  \int_0^T\int_\Omega|\na V|^2 dxdt \le C + C\int_0^T\int_\Omega
  (n^2+p^2+D^2)dxdt,
\end{align*}
where $C>0$ is a constant independent of $(n,p,D,\sigma)$. Next, we use the test function $n-\sigma\bar{n}$ in the weak formulation of \eqref{2.linn} and take into account that $S_k^1(n)\ge c(k)>0$ and $S_k^2(n)\ge 1$. Then, with the Young inequality,
\begin{align*}
  \frac12\int_\Omega&(n(t)-\sigma\bar{n})^2 dx
  - \frac12\int_\Omega(n^I-\sigma\bar{n})^2 dx
  + \int_0^t\int_\Omega|\na n|^2 dxds \\
  &\le C + C(k)\int_0^t\int_\Omega|\na V|^2 dxds 
  \le C + C\int_0^t\int_\Omega(n^2+p^2+D^2)dxds.
\end{align*}
We derive similar estimates when using $p-\sigma\bar{p}$ and $D$ in the weak formulations of \eqref{2.linp} and \eqref{2.linD}, respectively. Adding these estimates yields
\begin{align*}
  \int_\Omega\big(n(t)^2+p(t)^2+D(t)^2\big)dx
  &+ \int_0^t\int_\Omega\big(|\na n|^2+|\na p|^2+|\na D|^2\big)dxds \\
  &\le C + C\int_0^t\int_\Omega(n^2+p^2+D^2)dxds.
\end{align*}
We deduce from Gronwall's lemma $\sigma$-uniform bounds for $(n,p,D)$ in $L^\infty(0,T;L^2(\Omega))\cap L^2(0,T;H^1(\Omega))$. 
\end{proof}

The bounds in Lemma \ref{lem.FP} imply uniform estimates for $(\pa_t n,\pa_t p,\pa_t D)$ in $L^2(0,T;H^1_D(\Omega)')$. By the Aubin--Lions lemma, the embedding $L^2(0,T;H^1(\Omega))\cap H^1(0,T;H^1_D(\Omega)')\hookrightarrow L^2(\Omega_T)$ is compact. Thus, $F:L^2(\Omega_T)^3\times[0,1]\to L^2(\Omega_T)^3$ is compact. The assumptions of the Leray--Schauder fixed-point theorem are satisfied, and we conclude the existence of a fixed point of $F(\cdot,\cdot,\cdot;1)$, i.e.\ a solution to \eqref{2.nk}--\eqref{2.Vk} and \eqref{1.ic}--\eqref{1.noflux2}. We summarize:

\begin{lemma}[Existence for the approximate problem]\label{lem.approx}
\sloppy Let Assumptions (A1)--(A4) hold. Then there exists a weak solution to \eqref{2.nk}--\eqref{2.Vk} with initial and boundary conditions \eqref{1.ic}--\eqref{1.noflux2}. 
\end{lemma}

The solution $(n_k,p_k,D_k)$ to \eqref{2.nk}--\eqref{2.Vk} is componentwise nonnegative. Indeed, using the test function $n_k^-=\min\{0,n_k\}$ in the weak formulation of \eqref{2.nk}, we have
\begin{align*}
  \frac12\int_\Omega(n_k^-)(t)^2 dx 
  + \int_0^t\int_\Omega S_k^1(n_k)|\na n_k^-|^2 dxds
  = \int_0^t\int_{\{n_k<0\}} T_k(n_k)\na V_k\cdot\na n_k dxds = 0,
\end{align*}
since $T_k(n_k)=0$ for $n_k<0$, showing that $n_k^-(t)=0$ and consequently $n_k(t)\ge 0$ for $t>0$. We note that the mass of the oxide vacancies is conserved,
\begin{align*}
  \int_\Omega D_k(t)dx = \int_\Omega D^I dx\quad\mbox{for }t>0,
\end{align*}
while this is generally not the case for the electron and hole densities because of the Dirichlet boundary conditions. 


\subsection{Approximate energy inequality}

We derive the discrete analog of the free energy inequality \eqref{1.dEdt}. Similarly as in \cite[Sec.~2.3]{JJZ23}, we define for $0<\delta<\F_{1/2}(0)$ the approximations
\begin{equation}\label{2.GHkdelta}
\begin{aligned}
  G_{k,\delta}(s) &= \int_{\F_{1/2}(0)}^s\int_{\F_{1/2}(0)}^y
  \frac{S_k^1(z)}{T_k(z)+\delta}dzdy, \quad
  \widetilde{g}_{k,\delta}(s) = \int_0^s
  \frac{S_k^1(y)}{\sqrt{T_k(y)+\delta}}dy, \\
  H_{k,\delta}(s) &= \int_{\F_{-1}(0)}^s\int_{\F_{-1}(0)}^y
  \frac{S_k^2(z)}{T_{k/(k+1)}(z)+\delta}dzdy, \quad
  \widetilde{h}_{k,\delta}(s) = \int_0^s
  \frac{S_k^2(y)}{\sqrt{T_{k/(k+1)}(y)+\delta}}dy.
\end{aligned}
\end{equation}
Recalling that $S_k^1(z)\to zg'(z)$, $S_k^2(z)\to zh'(z)$, and $T_k(z)\to z$ pointwise as $k\to\infty$, the functions $G_{k,\delta}$ and $H_{k,\delta}$ approximate the anti-derivatives of $g$ and $h$, respectively (see \eqref{1.GH}), while $\widetilde{g}_{k,\delta}$ and $\widetilde{h}_{k,\delta}$ approximate
\begin{align*}
  \widetilde{g}(s) &:= \int_{\F_{1/2}(0)}^{s}\sqrt{z}g'(z)dz, \\
  \widetilde{h}(s) &:= \int_{\F_{-1}(0)}^s\sqrt{z}h'(z)dz
  = 2\tanh^{-1}(\sqrt{s}) - 2\tanh^{-1}(1/\sqrt{2}),
\end{align*}
respectively, since $\F_{-1}(0)=1/2$. These definitions yield the following chain rules:
\begin{equation}\label{2.chainrule}
\begin{aligned}
  S_k^1(n_k)\na n_k &= \sqrt{T_k(n_k)+\delta}
  \na \widetilde{g}_{k,\delta}(n_k), \\ 
  S_k^2(D_k)\na D_k &= \sqrt{T_{k/(k+1)}(D_k)+\delta}
  \na \widetilde{h}_{k,\delta}(D_k), 
\end{aligned}
\end{equation}
and similarly for $p_k$ instead of $n_k$. They are the truncated analogs of the chain rules $n_k g'(n_k)\na n_k=\sqrt{n_k}\na\widetilde{g}(n_k)$ and $D_k h'(n_k)\na D_k=\sqrt{D_k}\na\widetilde{h}(D_k)$. Choosing $\delta=0$ in \eqref{2.chainrule}, we see that the approximate fluxes can be formulated as
\begin{equation}\label{2.fluxk}
\begin{aligned}
  S_k^1(n_k)\na n_k - T_k(n_k)\na V_k 
  &= \sqrt{T_k(n_k)}\big(\na\widetilde{g}_k(n_k)
  - \sqrt{T_k(n_k)}\na V_k\big), \\
  S_k^1(p_k)\na p_k + T_k(p_k)\na V_k 
  &= \sqrt{T_k(p_k)}\big(\na\widetilde{g}_k(p_k)
  + \sqrt{T_k(p_k)}\na V_k\big), \\
  S_k^2(D_k)\na D_k + T_k(D_k)\na V_k 
  &= \sqrt{T_k(D_k)}\big(\na\widetilde{h}_k(D_k)
  + \sqrt{T_k(D_k)}\na V_k\big).
\end{aligned}
\end{equation}

Next, we define for $s,\bar{s}\ge 0$ the approximate relative energies
\begin{align}\label{2.relener}
  \mathcal{G}_{k,\delta}(s|\bar{s}) 
  = G_{k,\delta}(s) - G_{k,\delta}(\bar{s})
  - G_{k,\delta}'(\bar{s})(s-\bar{s}), \quad
  \mathcal{H}_{k,\delta}(s) = H_{k,\delta}(s) + s\bar{V}
\end{align}
and the approximate free energy
\begin{align}\label{2.Ek}
  E_{k,\delta}(n_k,p_k,D_k,V_k) = \int_\Omega\bigg(
  \mathcal{G}_{k,\delta}(n_k|\bar{n}) 
  + \mathcal{G}_{k,\delta}(p_k|\bar{p}) + \mathcal{H}_{k,\delta}(D_k)
  + \frac{\lambda^2}{2}|\na(V-\bar{V})|^2\bigg)dx.
\end{align}
We set
\begin{align}
  E_{k,\delta}^I &:= E_{k,\delta}(n^I,p^I,D^I,V^I), \label{2.EkI} \\
  \Lambda_{k,\delta} &:= 2\|\na(G'_{k,\delta}(\bar{n})-\bar{V})
  \|_{L^\infty(\Omega)}^2 + 2\|\na(G'_{k,\delta}(\bar{p})+\bar{V})
  \|_{L^\infty(\Omega)}^2. \label{2.Lambdak}
\end{align}
For the derivation of the approximate free energy inequality, we need the following lemma.

\begin{lemma}\label{lem.T53}
There exists a constant $C>0$ such that for any $k,\delta>0$ satisfying $0<\delta<\F_{1/2}(0)<k$,
\begin{align*}
  T_k(s)^{5/3} \le C(1+G_{k,\delta}(s)) \quad\mbox{for }s>0.
\end{align*}
\end{lemma}

\begin{proof}
Let $0<s\le\F_{1/2}(0)$. Since $s<k$, we have $T_k(s)^{5/3}=s^{5/3}\le \F_{1/2}(0)^{5/3}\le C$. Next, let $\F_{1/2}(0)<s\le k$. Then $S_k^1(s)=sg(s)$ and, by Lemma \ref{lem.Dg},
\begin{align*}
  G_{k,\delta}(s) &= \int_{\F_{1/2}(0)}^s\int_{\F_{1/2}(0)}^y
  \frac{zg'(z)}{z+\delta}dzdy 
  \ge C\int_{\F_{1/2}(0)}^s\int_{\F_{1/2}(0)}^y
  \frac{z(z^{-1}+z^{-1/3})}{z+\delta}dzdy \\
  &\ge \frac{C}{2}\int_{\F_{1/2}(0)}^s\int_{\F_{1/2}(0)}^y z^{-1/3}dzdy,
\end{align*}
since $\delta<\F_{1/2}(0)\le z$. An integration of the right-hand side leads to
\begin{align*}
  G_{k,\delta}(s) \ge \frac{3C}{4}\bigg(\frac35 s^{5/3}
  - \F_{1/2}(0)^{2/3}s + \frac25\F_{1/2}(0)^{5/3}\bigg)
  \ge CT_k(s)^{5/3} - C.
\end{align*}
Finally, if $s>k$, we have $T_k(s)^{5/3}=k^{5/3}\le CG_{k,\delta}(k)\le CG_{k,\delta}(s)$. This finishes the proof.
\end{proof}

\begin{lemma}[Approximate free energy inequality for $E_{k,\delta}$]
\label{lem.Ekdelta}
Let Assumptions (A1)--(A4) hold and let $(n_k,p_k,D_k,V_k)$ be the weak solution constructed in Lemma \ref{lem.approx}. Then, for all $0<t<T$,
\begin{align}\label{2.Ekdeltaineq}
  E_{k,\delta}(n_k,p_k,D_k,V_k)(t) 
  &+ \frac12\int_0^t\int_\Omega\big|\na \widetilde{g}_{k,\delta}(n_k)
  - \sqrt{T_k(n_k)+\delta}\na V_k\big|^2 dxds \\
  &+ \frac12\int_0^t\int_\Omega\big|\na \widetilde{g}_{k,\delta}(p_k)
  + \sqrt{T_k(p_k)+\delta}\na V_k\big|^2 dxds \nonumber \\
  &+ \frac12\int_0^t\int_\Omega\big|\na \widetilde{h}_{k,\delta}(D_k)
  + (T_{k/(k+1)}(D_k)+\delta)^{1/2}\na V_k\big|^2 dxds \nonumber \\
  \le\,& C(E_{k,\delta}^I,\Lambda_{k,\delta},T). \nonumber 
\end{align}
The constant $C(E_{k,\delta}^I,\Lambda_{k,\delta},T)\ge 0$ vanishes if $\Lambda_{k,\delta}=0$ and $\delta=0$.
\end{lemma}

\begin{proof}
We use the test function $G'_{k,\delta}(n_k)-G'_{k,\delta}(\bar{n})
- V_k + \bar{V}$ (see definition \eqref{2.GHkdelta}) in the weak formulation of \eqref{2.nk}:
\begin{align*}
  \big\langle \pa_t & n_k, G'_{k,\delta}(n_k)-G'_{k,\delta}(\bar{n})
  - V_k + \bar{V}\big\rangle \\
  &= -\int_\Omega\big(S_k^1(n_k)\na n_k - T_k(n_k)\na V_k\big)
  \cdot\na\big((G'_{k,\delta}(n_k) - V_k) - (G'_{k,\delta}(\bar{n}) 
  - \bar{V})\big)dx.
\end{align*}
We use the identities
\begin{align*}
  \big\langle \pa_t n_k, G'_{k,\delta}(n_k)-G'_{k,\delta}(\bar{n})
  \big\rangle &= \frac{d}{dt}\int_\Omega\big(G_{k,\delta}(n_k)
  - G_{k,\delta}(\bar{n}) - G'_{k,\delta}(\bar{n})(n_k-\bar{n})\big)dx, \\
  S_k^1(n_k)\na n_k - T_k(n_k)\na V_k
  &= \sqrt{T_k(n_k)+\delta}\big(\na \widetilde{g}_{k,\delta}(n_k)
  - \sqrt{T_k(n_k)+\delta}\na V_k\big) + \delta \na V_k, \\
  \na(G'_{k,\delta}(n_k)-V_k) &= \frac{\na \widetilde{g}_{k,\delta}(n_k)
  - \sqrt{T_k(n_k)+\delta}\na V_k}{\sqrt{T_k(n_k)+\delta}},
\end{align*}
which follow from the chain rules \eqref{2.chainrule}, to obtain
\begin{align}\label{2.En}
  \frac{d}{dt}&\int_\Omega\big(G_{k,\delta}(n_k)
  - G_{k,\delta}(\bar{n}) - G'_{k,\delta}(\bar{n})(n_k-\bar{n})\big)dx
  - \big\langle \pa_t n_k,V_k-\bar{V}\big\rangle \\
  &= -\int_\Omega\big|\na \widetilde{g}_{k,\delta}(n_k) 
  - \sqrt{T_k(n_k)+\delta}\na V_k\big|^2 dx \nonumber \\
  &\phantom{xx}- \int_\Omega\frac{\delta}{\sqrt{T_k(n_k)+\delta}}\na V_k
  \cdot\big(\na \widetilde{g}_{k,\delta}(n_k)
  - \sqrt{T_k(n_k)+\delta}\na V_k\big)dx \nonumber \\
  &\phantom{xx}+ \int_\Omega\sqrt{T_k(n_k)+\delta}
  \big(\na\widetilde{g}_{k,\delta}(n_k) 
  - \sqrt{T_k(n_k)+\delta}\na V_k\big)
  \cdot\na(G'_{k,\delta}(\bar{n})-\bar{V})dx \nonumber \\
  &\phantom{xx}
  + \delta\int_\Omega\na V_k\cdot\na(G'_{k,\delta}(\bar{n})-\bar{V})dx 
  \nonumber \\
  &\le -\frac12\int_\Omega\big|\na \widetilde{g}_{k,\delta}(n_k) 
  - \sqrt{T_k(n_k)+\delta}\na V_k\big|^2 dx
  + \frac{\delta}{2}\int_\Omega
  |\na(G'_{k,\delta}(\bar{n})-\bar{V})|^2 dx \nonumber \\
  &\phantom{xx} + \frac{3\delta}{2}\int_\Omega|\na V_k|^2 dx
  + \|\na(G'_{k,\delta}(\bar{n})-\bar{V})\|_{L^\infty(\Omega)}^2
  \int_\Omega(T_k(n_k)+\delta)dx, \nonumber 
\end{align}
where we used the inequality $\delta/(\sqrt{T_k(n_k)+\delta})\le\sqrt{\delta}$ and Young's inequality in the last step. Similarly, the test function $G'_{k,\delta}(p_k)-G'_{k,\delta}(\bar{p})
+ V_k - \bar{V}$ in the weak formulation of \eqref{2.pk} leads to
\begin{align*}
  \frac{d}{dt}&\int_\Omega\big(G_{k,\delta}(p_k)
  - G_{k,\delta}(\bar{p}) - G'_{k,\delta}(\bar{p})(p_k-\bar{p})\big)dx
  + \big\langle \pa_t p_k,V_k-\bar{V}\big\rangle \\
  &\le -\frac12\int_\Omega\big|\na \widetilde{g}_{k,\delta}(p_k) 
  + \sqrt{T_k(p_k)+\delta}\na V_k\big|^2 dx 
  + \frac{\delta}{2}\int_\Omega
  |\na(G'_{k,\delta}(\bar{p})+\bar{V})|^2 dx \nonumber \\
  &\phantom{xx} + \frac{3\delta}{2}\int_\Omega|\na V_k|^2 dx
  + \|\na(G'_{k,\delta}(\bar{p})+\bar{V})\|_{L^\infty(\Omega)}^2
  \int_\Omega(T_k(p_k)+\delta)dx. \nonumber 
\end{align*}
Similarly, with the test function $H'_{k,\delta}(D_k)+V_k$ in the weak formulation of \eqref{2.Dk},
\begin{align}\label{2.ED}
  \frac{d}{dt}&\int_\Omega \mathcal{H}_{k,\delta}(D_k)dx
  + \langle\pa_t D_k,V_k-\bar{V}\rangle 
  = \langle\pa_t D_k,H'_{k,\delta}(D_k) + V_k\rangle \\
  &\le -\frac12\int_\Omega\big|\na \widetilde{h}_{k,\delta}(D_k) 
  + (T_{k/(k+1)}(D_k)+\delta)^{1/2}\na V_k\big|^2 dx  
  + \frac{\delta}{2}\int_\Omega|\na V_k|^2 dx. \nonumber 
\end{align}

We add \eqref{2.En}--\eqref{2.ED} and take into account definition \eqref{2.relener} of the relative energies and definition \eqref{2.Lambdak} of $\Lambda_{k,\delta}$:
\begin{align}\label{2.aux}
  \frac{d}{dt}\int_\Omega&\big(\mathcal{G}_{k,\delta}(n_k|\bar{n}) 
  + \mathcal{G}_{k,\delta}(p_k|\bar{p})
  + \mathcal{H}_{k,\delta}(D_k)\big)dx
  - \big\langle\pa_t(n_k-p_k-D_k),V_k-\bar{V}\big\rangle \\
  &+ \frac12\int_\Omega\big|\na \widetilde{g}_{k,\delta}(n_k) 
  - \sqrt{T_k(n_k)+\delta}\na V_k\big|^2 dx \nonumber \\
  &+ \frac12\int_\Omega\big|\na \widetilde{g}_{k,\delta}(p_k) 
  + \sqrt{T_k(p_k)+\delta}\na V_k\big|^2 dx \nonumber \\
  &+ \frac12\int_\Omega\big|\na \widetilde{h}_{k,\delta}(D_k) 
  + (T_{k/(k+1)}(D_k)+\delta)^{1/2}\na V_k\big|^2 dx \nonumber \\
  \le\, & \Lambda_{k,\delta}\int_\Omega(T_k(n_k)+T_k(p_k) + 2\delta)dx
  + \frac{7\delta}{2}\int_\Omega|\na V_k|^2 dx
  + \frac{\delta}{2}|\Omega|\Lambda_{k,\delta}. \nonumber 
\end{align}
In view of Poisson's equation \eqref{2.Vk}, the last term in the first line of \eqref{2.aux} can be written as
\begin{align*}
  -\big\langle\pa_t(n_k-p_k-D_k),V_k-\bar{V}\big\rangle
  = -\lambda^2\langle\pa_t\Delta V_k,V_k-\bar{V}\rangle 
  = \frac{\lambda^2}{2}\frac{d}{dt}\int_\Omega|\na(V_k-\bar{V})|^2 dx.
\end{align*}
Integrating \eqref{2.aux} over $(0,t)$ and taking into account 
\begin{align*}
  \int_\Omega|\na V_k|^2 dx \le 2\int_\Omega|\na(V_k-\bar{V})|^2 dx
  + 2\int_\Omega|\na \bar{V}|^2 dx
\end{align*}
as well as definitions \eqref{2.Ek} for $E_{k,\delta}$ and \eqref{2.EkI} for $E_{k,\delta}^I$, we arrive at
\begin{align*}
  E_{k,\delta}&(n_k,p_k,D_k,V_k)(t)
  + \frac12\int_0^t\int_\Omega\big|\na \widetilde{g}_{k,\delta}(n_k) 
  - \sqrt{T_k(n_k)+\delta}\na V_k\big|^2 dxds \\
  &\phantom{xx}+ \frac12\int_0^t\int_\Omega
  \big|\na\widetilde{g}_{k,\delta}(p_k) 
  + \sqrt{T_k(p_k)+\delta}\na V_k\big|^2 dxds \nonumber \\
  &\phantom{xx}+ \frac12\int_0^t\int_\Omega
  \big|\na \widetilde{h}_{k,\delta}(D_k) 
  + (T_{k/(k+1)}(D_k)+\delta)^{1/2}\na V_k\big|^2 dxds \nonumber \\
  &\le E_{k,\delta}^I 
  + \Lambda_{k,\delta}\int_0^t \int_\Omega
  (T_k(n_k)+T_k(p_k) + 2\delta)dxds 
  \nonumber \\
  &\phantom{xx}+ 7\delta\int_0^t \int_\Omega
  |\na (V_k-\bar{V})|^2 dxds
  + 7\delta\int_0^t \int_\Omega|\na\bar{V}|^2 dxds
  + \frac{\delta}{2}|\Omega|\Lambda_{k,\delta}t. \nonumber 
\end{align*}
We conclude from Young's inequality and Lemma \ref{lem.T53} that
\begin{align*}
  T_k(n_k) \le C + T_k(n_k)^{5/3} \le C + CG_{k,\delta}(n_k),
\end{align*}
and consequently, the second term on the right-hand side can be replaced by 
$$
  C_1\Lambda_{k,\delta}\int_0^t E_{k,\delta}(n_k,p_k,D_k,V_k)ds  
  + C(\Omega)(\delta+1)\Lambda_{k,\delta}t.
$$
Thus, applying Gronwall's lemma,
\begin{align*}
  E_{k,\delta}(n_k,p_k,D_k,V_k)(t)
  &+ \frac12\int_0^t\int_\Omega\big|\na \widetilde{g}_{k,\delta}(n_k) 
  - \sqrt{T_k(n_k)+\delta}\na V_k\big|^2 dxds \\
  &+ \frac12\int_0^t\int_\Omega
  \big|\na \widetilde{g}_{k,\delta}(p_k) 
  + \sqrt{T_k(p_k)+\delta}\na V_k\big|^2 dxds \nonumber \\
  &+ \frac12\int_0^t\int_\Omega
  \big|\na \widetilde{h}_{k,\delta}(D_k) 
  + (T_{k/(k+1)}(D_k)+\delta)^{1/2}\na V_k\big|^2 dxds \nonumber \\
  \le\,& \big(E_{k,\delta}^I + C(\Omega)(\delta+1)\Lambda_{k,\delta}t\big)
  \exp(C_1\Lambda_{k,\delta}t).
\end{align*}
This proves the lemma.
\end{proof}

\subsection{Limit $\delta\to 0$}

We set
\begin{align*}
  & \widetilde{g}_k(s) = \widetilde{g}_{k,0}(s), 
  \quad G_k(s) = G_{k,0}(s), \quad 
  \widetilde{h}_k(s) = \widetilde{h}_{k,0}(s), 
  \quad H_k(s) = H_{k,0}(s), \\
  & \mathcal{G}_k(s|\bar{s}) = \mathcal{G}_{k,0}(s|\bar{s}), \quad
  \mathcal{H}_k(s|\bar{s}) = \mathcal{H}_{k,0}(s|\bar{s})
\end{align*}
and introduce 
\begin{align*}
  E_k(n_k,p_k,D_k,V_k) &= \int_\Omega\bigg(\mathcal{G}_k(n_k|\bar{n})
  + \mathcal{G}_k(p_k|\bar{p}) + \mathcal{H}_k(V_k) 
  + \frac{\lambda^2}{2}|\na(V_k-\bar{V})|^2\bigg)dx, \\
  E_k^I &= E_k(n^I,p^I,D^I,V^I), \quad \Lambda_k = \Lambda_{k,0}.
\end{align*}

\begin{lemma}[Approximate free energy inequality for $E_k$]
Under the assumptions of Lemma \ref{lem.Ekdelta}, it holds for all $0<t<T$ that
\begin{align}\label{2.Ekineq}
  E_{k}&(n_k,p_k,D_k,V_k)(t)
  + \frac12\int_0^t\int_\Omega\big|\na \widetilde{g}_{k}(n_k) 
  - \sqrt{T_k(n_k)}\na V_k\big|^2 dxds \\
  &\phantom{xx}+ \frac12\int_0^t\int_\Omega
  \big|\na \widetilde{g}_{k}(p_k) 
  + \sqrt{T_k(p_k)}\na V_k\big|^2 dxds \nonumber \\
  &\phantom{xx}+ \frac12\int_0^t\int_\Omega
  \big|\na\widetilde{h}_{k}(D_k)
  + T_{k/(k+1)}(D_k)^{1/2}\na V_k\big|^2 dxds
  \le C(E_k^I,\Lambda_k,T), \nonumber 
\end{align}
and the constant $C(E_k^I,\Lambda_k,T)\ge 0$ vanishes if $\Lambda_k=0$. 
\end{lemma}

\begin{proof}
The proof essentially follows from the monotone and dominated convergence theorems as in the proof of Lemma 2.3 of \cite{JJZ23}. The only difference is the treatment of the functions $\widetilde{g}_{k,\delta}$, $G_{k,\delta}$, $\widetilde{h}_{k,\delta}$, and $H_{k,\delta}$. In fact, by the monotone convergence theorem, $\widetilde{g}_{k,\delta}(n_k)\to \widetilde{g}_k(n_k)$, $G_{k,\delta}(n_k)\to G_k(n_k)$, $\widetilde{h}_{k,\delta}(D_k)\to \widetilde{h}_k(D_k)$, and $H_{k,\delta}(D_k)\to H_k(D_k)$ a.e.\ in $\Omega_T$ as $\delta\to 0$.
We derive upper bounds for the limit functions. Let $s>k$. Then, using Lemma \ref{lem.Dg},
\begin{align*}
  \widetilde{g}_k(s) &= \int_0^{s}\widetilde{g}'_k(z)dz
  = \int_0^k\sqrt{z}g'(z)dz + \int_k^{s}k^{1/6}z^{1/3}g'(z)dz \\
  &\le C\int_0^k(z^{1/6}+z^{-1/2})dz + Ck^{1/6}\int_k^s(1+z^{-2/3})dz
  \le C(k)(s+1),
\end{align*}
and this inequality also holds for any $s\ge 0$. We obtain for $s>k/(k+1)$:
\begin{align*}
  \widetilde{h}_k(s) = \int_0^{k/(k+1)}\frac{dz}{\sqrt{z}(1-z)}
  + \int_{k/(k+1)}^s(k+1)\sqrt{\frac{k+1}{k}}dz \le C(k)(s+1),  
\end{align*}
and this bound holds in fact for all $s\ge 0$. Similar arguments lead to
\begin{align*}
  G_{k,\delta}(s)\le C(k)(s^2+1), \quad
  H_{k,\delta}(s)\le C(k)(s^2+1)\quad\mbox{for }s\ge 0. 
\end{align*}
The approximate free energy inequality \eqref{2.Ekdeltaineq} implies that $n_k,p_k,D_k$ are bounded in $L^2(\Omega_T)$ uniformly in $\delta$. Hence, we can apply the dominated convergence theorem to find that
\begin{align*}
  \widetilde{g}_{k,\delta}(n_k)\to \widetilde{g}_k(n_k), \quad 
  \widetilde{h}_{k,\delta}(D_k)\to \widetilde{h}_k(D_k)
  &\quad\mbox{strongly in }L^2(\Omega_T), \\
  G_{k,\delta}(n_k)\to G_k(n_k), \quad H_{k,\delta}(D_k)\to H_k(D_k)
  &\quad\mbox{strongly in }L^1(\Omega_T).
\end{align*}
The sequence $(\na\widetilde{g}_{k,\delta}(n_k)-\sqrt{T_k(n_k)+\delta}\na V_k)_\delta$ is uniformly bounded in $L^2(\Omega_T)$. Therefore, there exists a subsequence that converges weakly in $L^2(\Omega_T)$ as $\delta\to 0$. The previous arguments allow us to identify the weak limit, showing the claim. The other terms in \eqref{2.Ekdeltaineq} can be treated in a similar way. The limit $\delta\to 0$ in \eqref{2.Ekdeltaineq} then proves \eqref{2.Ekineq}.
\end{proof}

\subsection{Uniform estimates}

We first show some inequalities relating $g$, $G_k$, $T_k$, and $\widetilde{g}$. 

\begin{lemma}\label{lem.g}
There exists $C>0$ such that for all $k>1$ and $s>0$,
\begin{align*}
  & g'(s)\le G_k''(s), \quad s^{5/3}\le C(G_k(s)+1), \quad
  T_k(s)^{7/6}\le C\widetilde{g}_k(s), \\ 
  & \widetilde{g}_k(s)^{10/7}\le C(G_k(s)+1), \quad 
  T_k(s)^{5/3}\le C(G_k(s)+1), \\
  & \widetilde{h}_k(s)\ge s^{-1/2}, \quad
  \widetilde{h}'_k(s)\ge 1.
\end{align*}
\end{lemma}

\begin{proof}
The first inequality follows from $G_k''(s)=g'(s)$ for $0<s\le k$ and
$G_k''(s)=(s/k)^{1/3}g'(s)\ge g'(s)$ for $s>k$. Since $g'(s)\sim s^{-1}+s^{-1/3}$ by Lemma \ref{lem.Dg},
\begin{align*}
  G_k(s) \ge C\int_{\F_{1/2}(0)}^s\int_{\F_{1/2}(0)}^y(z^{-1}+z^{-1/3})
  dzdy \ge C(s^{5/3}-1),
\end{align*}
which proves the second inequality. For the third one, let $0<s\le k$. Then, again by Lemma \ref{lem.Dg},
\begin{align*}
  \widetilde{g}'_k(z) = \sqrt{z}g'(z) \ge C(z^{-1/2}+z^{1/6}) 
  \ge Cz^{1/6}.
\end{align*}
We integrate this inequality over $z\in(0,s)$ to find that $\widetilde{g}_k(s)=\widetilde{g}_k(s)-\widetilde{g}_k(0)\ge Cs^{7/6} = CT_k(s)^{7/6}$. If $s>k$, we compute
\begin{align*}
  \widetilde{g}_k(s)\ge \int_0^k\frac{yg'(y)}{\sqrt{T_k(y)}}dy
  = \int_0^k\sqrt{y}g'(y)dy \ge C\int_0^k y^{1/6}dy = Ck^{7/6}
  = CT_k(s)^{7/6}.
\end{align*}
We turn to the fourth inequality. By Lemma \ref{lem.Dg}, we have for $0<s\le k$,
\begin{align*}
  \widetilde{g}_k(s) = \int_0^s \sqrt{y}g'(y)dy 
  \le C\int_0^s(y^{-1/2}+y^{1/6})dy
  \le C(s^{7/6}+1).
\end{align*}
Furthermore, if $s>k$,
\begin{align*}
 \widetilde{g}_k(s) &= \int_0^k\sqrt{y}g'(y)dy 
  + \int_k^s k^{1/6}y^{1/3}g'(y)dy \\
  &\le C(k^{1/2}+k^{7/6}) + Ck^{1/6}(s^{1/3}+s) \le C(s^{7/6}+1),
\end{align*}
and the conclusion follows after raising the inequality to the power $10/7$ and using the second inequality. The fifth inequality follows from the third and fourth ones since $T_k(s)^{5/3}\le C\widetilde{g}_k(s)^{10/7}\le C(G_k(s)+1)$.

To estimate $\widetilde{h}'$, we observe that $h'(s)=1/(s(1-s))$ and hence
$\widetilde{h}'_k(s)=1/(\sqrt{s}(1-s))\ge s^{-1/2}$ for $s<k/(k+1)$ and $\widetilde{h}'_k(s)=(1+k)/k\ge s^{-1}\ge s^{-1/2}$ for $k/(k+1)\le s<1$. Moreover, in both cases, $\widetilde{h}'(s)\ge 1$. This proves the inequalities for $\widetilde{h}'_k$.
\end{proof}

The previous lemma and the approximate energy inequality \eqref{2.Ekineq} lead to the following a priori estimates.

\begin{lemma}[Uniform estimates I]\label{lem.est1}
There exists a constant $C>0$ such that for all $k\in\N$,
\begin{align*}
  \|T_k(n_k)\|_{L^\infty(0,T;L^{5/3}(\Omega))}
  + \|T_k(p_k)\|_{L^\infty(0,T;L^{5/3}(\Omega))} &\le C, \\
  \|n_k\|_{L^\infty(0,T;L^{5/3}(\Omega))}
  + \|p_k\|_{L^\infty(0,T;L^{5/3}(\Omega))} &\le C, \\
  \|\widetilde{g}_k(n_k)\|_{L^\infty(0,T;L^{10/7}(\Omega))}
  + \|\widetilde{g}_k(p_k)\|_{L^\infty(0,T;L^{10/7}(\Omega))} &\le C, \\
  \|\sqrt{T_k(n_k)}\na V_k\|_{L^\infty(0,T;L^{5/4}(\Omega))}
  + \|\sqrt{T_k(p_k)}\na V_k\|_{L^\infty(0,T;L^{5/4}(\Omega))} &\le C, \\
  \|\na\widetilde{g}_k(n_k)\|_{L^2(0,T;L^{5/4}(\Omega))}
  + \|\na\widetilde{g}_k(p_k)\|_{L^2(0,T;L^{5/4}(\Omega))} &\le C, \\
  \|\na n_k\|_{L^2(0,T;L^{5/4}(\Omega))}
  + \|\na p_k\|_{L^2(0,T;L^{5/4}(\Omega))} &\le C, \\
  \|\na\widetilde{h}_k(D_k)\|_{L^2(\Omega_T)} 
  + \|(T_{k/(k+1)}(D_k))^{1/2}\na V_k\|_{L^\infty(0,T;L^2(\Omega))} 
  &\le C, \\
  \|\na D_k\|_{L^2(\Omega_T)} 
  + \|\na\sqrt{D_k}\|_{L^2(\Omega_T)} &\le C.
\end{align*}
\end{lemma}

\begin{proof}
The approximate energy inequality \eqref{2.Ekineq} shows that
$(G_k(n_k))$ and $(G_k(p_k))$ are bounded in $L^\infty(0,T;L^1(\Omega))$. By Lemma \ref{lem.g}, this yields a uniform bound for $T_k(n_k)$, $T_k(p_k)$ and $n_k$, $p_k$ in $L^\infty(0,T;L^{5/3}(\Omega))$ and for $\widetilde{g}_k(n_k)$, $\widetilde{g}_k(p_k)$ in $L^\infty(0,T;L^{10/7}(\Omega))$. The energy estimate \eqref{2.Ekineq} implies a uniform bound for $\na V_k$ in $L^\infty(0,T;L^2(\Omega))$. Consequently, using H\"older's inequality,
\begin{align*}
  \|\sqrt{T_k(n_k)}\na V_k\|_{L^\infty(0,T;L^{5/4}(\Omega))}
  \le \|\sqrt{T_k(n_k)}\|_{L^\infty(0,T;L^{10/3}(\Omega))}
  \|\na V_k\|_{L^\infty(0,T;L^2(\Omega))} \le C,
\end{align*}
and similarly for $\sqrt{T_k(p_k)}\na V_k$. Then we deduce from the $L^2(\Omega_T)$ bound for $\na\widetilde{g}_k(n_k)-\sqrt{T_k(n_k)}\na V_k$ that $\na \widetilde{g}_k(n_k)$ is uniformly bounded in $L^2(0,T;L^{5/4}(\Omega))$. It follows from Lemma \ref{lem.Dg} that
$\widetilde{g}_k'(s)=\sqrt{s}g'(s)\ge C(s^{1/6}+s^{-1/2})\ge C$ for $0<s\le k$ and $\widetilde{g}_k'(s) = k^{1/6}s^{1/3}g'(s)\ge k^{1/6}(1+s^{-2/3})\ge 1$ for $s>k$. Thus, $\widetilde{g}_k'$ is bounded from below by a positive constant. Then the bounds for $\na \widetilde{g}_k(n_k)$ and $\na \widetilde{g}_k(p_k)$ imply the same bounds for $\na n_k$ and $\na p_k$.

We turn to the estimates for $D_k$. Since $T_{k/(k+1)}(D_k)<1$, we infer from the energy inequality \eqref{2.Ekineq} that $(T_{k/(k+1)}(D_k)^{1/2}\na V_k)$ is bounded in $L^\infty(0,T;L^2(\Omega))$ and consequently, $\na \widetilde{h}_k(D_k)$ is uniformly bounded in $L^2(\Omega_T)$. We deduce from $\widetilde{h}_k'(s)\ge s^{-1/2}$ and $\widetilde{h}_k'(s)\ge 1$ for $s>0$ (see Lemma \ref{lem.g}) that $\na\sqrt{D_k}$ and $\na D_k$ are uniformly bounded in $L^2(\Omega_T)$.
\end{proof}

We can improve the regularity stated in Lemma \ref{lem.est1} by using the inequality $T_k(s)^{7/6}\le C\widetilde{g}_k(s)$ and an iteration argument.

\begin{lemma}[Uniform estimates II]\label{lem.est2}
Let $q,r\le\infty$ if $d=1$, $q,r<\infty$ if $d=2$, and $q<8$, $r<24/7$ if $d=3$. There exists a constant $C>0$ such that for all $k\in\N$,
\begin{align*}
  \|\sqrt{T_k(n_k)}\|_{L^{14/3}(0,T;L^q(\Omega))}
  + \|\sqrt{T_k(n_k)}\|_{L^{14/3}(0,T;L^q(\Omega))} &\le C, \\
  \|\sqrt{T_k(n_k)}\na V_k\|_{L^{14/3}(0,T;L^{r}(\Omega))}
  + \|\sqrt{T_k(n_k)}\na V_k\|_{L^{14/3}(0,T;L^{r}(\Omega))} 
  & \le C,\\
  \|\widetilde{g}_k(n_k)\|_{L^2(0,T;L^{r}(\Omega))}
  + \|\widetilde{g}_k(p_k)\|_{L^2(0,T;L^{r}(\Omega))} &\le C,\\
  \|\na \widetilde{g}_k(n_k)\|_{L^2(0,T;L^{2q/(2+q)}(\Omega))}
  + \|\na\widetilde{g}_k(p_k)\|_{L^2(0,T;L^{2q/(2+q)}(\Omega))} &\le C,\\
  \|\na n_k\|_{L^2(0,T;L^{2q/(2+q)}(\Omega))}
  + \|\na p_k\|_{L^2(0,T;L^{2q/(2+q)}(\Omega))} &\le C
\end{align*}
Observe that $2q/(2+q)<8/5$ if $d=3$.
\end{lemma}

\begin{proof}
We assume that $(\sqrt{T_k(n_k)})$ is bounded in $L^{14/3}(0,T;L^{q_{m}}(\Omega))$ and $(\widetilde{g}_k(n_k))$ is bound\-ed in $L^2(0,T;L^{r_{m}}(\Omega))$ for some numbers $q_m$, $r_m\ge 1$ with $m\in\N$. By Lemma \ref{lem.est1}, we have $q_1=10/3$ and $r_1=10/7$. We estimate
\begin{align}\label{2.sqrtTnaV}
  \|\sqrt{T_k(n_k)}\na V_k\|_{L^{14/3}(0,T;L^a(\Omega))}
  \le \|\sqrt{T_k(n_k)}\|_{L^{14/3}(0,T;L^{q_m}(\Omega))}
  \|\na V_k\|_{L^\infty(0,T;L^2(\Omega))} \le C,
\end{align} 
where $1/a = 1/q_m+1/2$. This shows that
\begin{align}\label{2.nag}
  \|\na \widetilde{g}_k(n_k)\|_{L^2(0,T;L^a(\Omega))}
  &\le \|\na \widetilde{g}_k(n_k)
  -\sqrt{T_k(n_k)}\na V_k\|_{L^2(0,T;L^a(\Omega))} \\
  &\phantom{xx}+ \|\sqrt{T_k(n_k)}\na V_k\|_{L^{2}(0,T;L^a(\Omega))} 
  \le C. \nonumber
\end{align}
Then the continuous embedding $W^{1,a}(\Omega)\hookrightarrow L^{r_{m+1}}(\Omega)$ with $1/r_{m+1}=1/a-1/d=1/q_m+1/2-1/d$ implies that
$\widetilde{g}_k(n_k)$ is uniformly bounded in $L^2(0,T;L^{r_{m+1}}(\Omega))$. We deduce from $\sqrt{T_k(s)}\le C\widetilde{g}_k(s)^{3/7}$ (see Lemma \ref{lem.g}) that $\sqrt{T_k(n_k)}$ is uniformly bounded in $L^{14/3}(0,T;L^{q_{m+1}}(\Omega))$, where $q_{m+1}=7r_{m+1}/3$. This leads to the recursion
\begin{align*}
  \frac{1}{q_{m+1}} = \frac{3}{7}\frac{1}{r_{m+1}}
  = \frac37\bigg(\frac{1}{q_m}+\frac12-\frac{1}{d}\bigg).
\end{align*}
The sequence $(1/q_m)$ is nonincreasing (if $d\le 4$) and bounded from below. Thus, it possesses the limit $q^*$ that satisfies
\begin{align*}
  \frac{1}{q^*} = \frac37\bigg(\frac{1}{q^*}+\frac12-\frac{1}{d}\bigg)
  \quad\mbox{and hence}\quad q^* = \frac{8d}{3(d-2)}.
\end{align*}
We can perform this recursion only a finite number of times as otherwise the powers of the embedding constant may diverge. Thus, 
$q<q^*=8$ if $d=3$ and, since $q_{m+1}=7q_m/3\to\infty$ as $m\to\infty$, $q<\infty$ if $d=2$. Furthermore, $1/r=1/q+1/2-1/d$, which gives $r<24/7$ if $d=3$ and $r<\infty$ if $d=2$. 

The uniform bound for $\sqrt{T_k(n_k)}\na V_k$ follows from \eqref{2.sqrtTnaV} because of $1/a=1/q+1/2=(2+q)/(2q)$. Estimate \eqref{2.nag} then implies the bound for $\na \widetilde{g}_k(n_k)$. We have shown in the proof of Lemma \ref{lem.est1} that $\widetilde{g}_k'$ is bounded from below by a positive constant such that $C|\na n_k|\le |\widetilde{g}'_k(n_k)\na n_k|=|\na \widetilde{g}_k(n_k)|$, proving the last uniform bound. Finally, the estimates for $p_k$ are proved analogously.
\end{proof}

The following bounds are needed for the Aubin--Lions lemma.

\begin{lemma}[Uniform estimates III]\label{lem.est3}
Under the assumptions of Lemma \ref{lem.est2}, there exists a constant $C>0$ such that for all $k\in\N$,
\begin{align*}
  \|n_k\|_{L^{2}(0,T;W^{1,2q/(2+q)}(\Omega))}
  + \|p_k\|_{L^{2}(0,T;W^{1,2q/(2+q)}(\Omega))} &\le C, \\
  \|\pa_t n_k\|_{L^{7/5}(0,T;W_D^{1,2q/(q+4)}(\Omega)')}
  + \|\pa_t p_k\|_{L^{7/5}(0,T;W_D^{1,2q/(q+4)}(\Omega)')} &\le C, \\
  \|D_k\|_{L^2(0,T;H^1(\Omega))} 
  + \|\pa_t D_k\|_{L^2(0,T;H^1(\Omega)')} &\le C.
\end{align*}
\end{lemma}

\begin{proof}
The bound for $n_k$ follows from the gradient bounds in Lemma \ref{lem.est2} and the bound for $n_k$ in $L^\infty(0,T;L^{5/3}(\Omega))$, which is a consequence of the energy inequality \eqref{2.Ekineq} and the inequality $s^{5/3}\le C(G_k(s)+1)$ from Lemma \ref{lem.g}. By the chain rule \eqref{2.chainrule} (for $\delta=0$), the evolution equation for $n_k$ reads as
\begin{align*}
  \pa_t n_k = \diver\big[\sqrt{T_k(n_k)}\big(\na \widetilde{g}_k(n_k)
  - \sqrt{T_k(n_k)}\na V_k\big)\big].
\end{align*}
The term $\sqrt{T_k(n_k)}$ is uniformly bounded in $L^{14/3}(0,T;L^q(\Omega))$, while $\na \widetilde{g}_k(n_k) - \sqrt{T_k(n_k)}\na V_k$ is uniformly bounded in $L^2(0,T;L^{2q/(2+q)}(\Omega))$. Hence, $(\pa_t n_k)$ is bounded in $L^{7/5}(0,T;$ $L^{2q/(4+q)}(\Omega))$ (we choose $q\ge 4$ to guarantee that $2q/(4+q)\ge 1$). The estimates for $p_k$ are shown in a similar way.

We turn to the bounds for $D_k$. We know from the energy inequality \eqref{2.Ekineq} that $(H_k(D_k))$ is bounded in $L^\infty(0,T;L^1(\Omega))$ and from Lemma \ref{lem.est1} that $(\na D_k)$ is bounded in $L^2(\Omega_T)$. Proceeding as in the proof of Lemma \ref{lem.g}, we infer that $S_k^2/T_{k/(k+1)}$ is bounded from below by a positive constant. This implies that $H_k(D_k)\ge CD_k^2$, which yields a bound for $(D_k)$ in $L^\infty(0,T;L^2(\Omega))$. This shows that $(D_k)$ is bounded in $L^2(0,T;H^1(\Omega))$. Finally, since $(T_{k/(k+1)}(D_k))$ is bounded in $L^\infty(\Omega_T)$, the sequence 
\begin{align*}
  (\na \widetilde{h}_k(D_k)+T_{k/(k+1)}(D_k)^{1/2}\na V_k)
\end{align*} 
is bounded in $L^2(\Omega_T)$. Therefore,
\begin{align*}
  \pa_t D_k = \diver\big[T_{k/(k+1)}(D_k)^{1/2}\big(\na\widetilde{h}_k(D_k)
  + T_{k/(k+1)}(D_k)^{1/2}\na V_k\big)\big]
\end{align*}
is uniformly bounded in $L^2(0,T;H^1(\Omega)')$, finishing the proof.
\end{proof}

\subsection{Limit $k\to\infty$ in the equations for $n_k$ and $p_k$}

Lemma \ref{lem.est3} and the compact embedding $W^{1,2q/(2+q)}(\Omega)\hookrightarrow L^r(\Omega)$ for $r<24/7$ (if $d\le 3$) allow us to apply the Aubin--Lions lemma to infer the existence of a subsequence that is not relabeled such that, as $k\to\infty$,
\begin{align*}
  (n_k,p_k)\to (n,p) 
  &\quad\mbox{strongly in }L^2(0,T;L^{r}(\Omega))^2, \\
  D_k\to D &\quad\mbox{strongly in }L^2(\Omega_T), \\
  (\na n_k,\na p_k)\rightharpoonup (\na n,\na p)
  &\quad\mbox{weakly in }L^2(0,T;L^{2q/(2+q)}(\Omega)), \\
  (\pa_t n_k,\pa_t p_k)\rightharpoonup (\pa_t n,\pa_t p)
  &\quad\mbox{weakly in }L^{7/5}(0,T;W_D^{1,2q/(q+4)}(\Omega)')^2, \\
  \pa_t D_k \rightharpoonup \pa_t D
  &\quad\mbox{weakly in }L^2(0,T;H^1(\Omega)').
\end{align*} 
The $L^\infty(0,T;L^{5/3}(\Omega))$ bound for $n_k$ from Lemma \ref{lem.est1} implies that
\begin{align*}
  \|T_k(n_k)-n_k\|_{L^1(\Omega_T)}
  = \int_0^T\int_{\{n_k>k\}}|k-n_k|dxdt
  \le \int_0^T\int_\Omega\frac{n_k^{5/3}}{k^{2/3}}dxdt
  \le \frac{C}{k^{2/3}}\to 0.
\end{align*}
This shows that $\sqrt{T_k(n_k)}-\sqrt{n_k}\to 0$ strongly in $L^2(\Omega_T)$ and in particular $\sqrt{T_k(n_k)}\to \sqrt{n}$ strongly in $L^2(\Omega_T)$. In fact, in view of the $L^\infty(0,T;L^{5/3}(\Omega))$ bound for $n_k$, we even have strong convergence for $\sqrt{T_k(n_k)}$ in $L^s(\Omega_T)$ for any $s<10/3$. The $L^2(\Omega_T)$ bound for $\na V_k$ shows that $\na V_k\rightharpoonup\na V$ weakly in $L^2(\Omega_T)$. Hence,
\begin{align*}
  \sqrt{T_k(n_k)}\na V_k\rightharpoonup\sqrt{n}\na V
  \quad\mbox{weakly in }L^1(\Omega_T).
\end{align*}
Thanks to Lemma \ref{lem.est2}, this convergence also holds in $L^{14/3}(0,T;L^{r}(\Omega))$ with $r<8/5$. Then, since $s<10/3$ and $q<8$ can be chosen in such a way that $1/s+(2+q)/(2q)<1$,
\begin{align}\label{2.convTV}
  T_k(n_k)\na V_k = \sqrt{T_k(n_k)}\cdot\sqrt{T_k(n_k)}\na V_k
  \rightharpoonup n\na V\quad\mbox{weakly in }L^1(\Omega_T).
\end{align}
Similarly, we have $T_k(p_k)\na V_k\rightharpoonup T_k(p_k)\na V_k$ weakly in $L^1(\Omega_T)$. 

Next, we prove the weak convergence of $(\sqrt{T_k(n_k)}\na\widetilde{g}_k(n_k))$.

\begin{lemma}\label{lem.limTnag}
It holds that, up to a subsequence, 
\begin{align*}
  \sqrt{T_k(n_k)}\big(\na\widetilde{g}_k(n_k)-\sqrt{T_k(n_k)}\na V_k\big)
  \rightharpoonup n\na g(n)-n\na V
  &\quad\mbox{weakly in }L^1(\Omega_T), \\
  \sqrt{T_k(p_k)}\big(\na\widetilde{g}_k(p_k)+\sqrt{T_k(p_k)}\na V_k\big)
  \rightharpoonup p\na g(p) + p\na V
  &\quad\mbox{weakly in }L^1(\Omega_T).
\end{align*}
\end{lemma}

\begin{proof}
First, we show that $\sqrt{T_k(n_k)}\widetilde{g}_k'(n_k)$ converges strongly. To this end, we estimate
\begin{align}\label{2.weak1}
  \big\|&\sqrt{T_k(n_k)}\widetilde{g}_k'(n_k) - ng'(n)
  \big\|_{L^2(0,T;L^{5}(\Omega))} \\
  &\le \big\|\sqrt{T_k(n_k)}\widetilde{g}_k'(n_k) - n_kg'(n_k)
  \big\|_{L^2(0,T;L^{5}(\Omega))}
  + \|n_kg'(n_k)-ng'(n)\|_{L^2(0,T;L^{5}(\Omega))}. \nonumber
\end{align}
The first integrand vanishes on $\{n_k\le k\}$. Therefore, on the set $\{n_k>k\}$, using Lemma \ref{lem.Dg},
\begin{align*}
  \big|&\sqrt{T_k(n_k)}\widetilde{g}_k'(n_k) - n_kg'(n_k)\big|^5
  = |(n_k^{2/3}-k^{2/3})n_k^{1/3}g'(n_k)|^{5} \\
  &\le |n_kg'(n_k)|^{5} 
  \le C|n_k(n_k^{-1}+n_k^{-1/3})|^{5} \le C(1+n_k^{10/3})
  \le C\bigg(\frac{n_k}{k}+\frac{n_k^{10/3+\eps}}{k^{\eps}}\bigg)
\end{align*}
for $\eps>0$ such that $10/3+\eps\le r<24/7$, which shows that
\begin{align*}
  \big\|\sqrt{T_k(n_k)}\widetilde{g}_k'(n_k) - n_kg'(n_k)
  \big\|_{L^2(0,T;L^{5}(\Omega))}^2
  \le \frac{C}{k^{2\eps}}\int_0^T\big(\|n_k\|_{L^1(\Omega)} 
  + \|n_k\|_{L^r(\Omega)}^{r}\big)^{2/5}dt\to 0,
\end{align*}
since $(n_k)$ is bounded in $L^2(0,T;L^r(\Omega))$. We show in Lemma \ref{lem.DzDg} that $|(zg'(z))'|=|g'(z)+zg''(z)|$ is bounded for $z>0$. Hence,
\begin{align*}
  |n_kg'(n_k)-ng'(n)|^r 
  = \bigg|\int_{n_k}^n\frac{d}{dz}(zg'(z))dz\bigg|^r
  \le C|n_k-n|^r.
\end{align*}
Then the strong convergence $n_k\to n$ in $L^2(0,T;L^r(\Omega))$ implies that 
\begin{align*}
  \|n_kg'(n_k)-ng'(n)\|_{L^2(0,T;L^r(\Omega))}\to 0.  
\end{align*}
We conclude from \eqref{2.weak1} that
\begin{align*}
  \sqrt{T_k(n_k)}\widetilde{g}_k'(n_k) \to ng'(n)
  \quad\mbox{strongly in }L^2(0,T;L^r(\Omega)).
\end{align*}
Consequently, the weak convergence $\na n_k\rightharpoonup\na n$ in $L^2(0,T;L^{2q/(2+q)}(\Omega))$ gives
\begin{align*}
  \sqrt{T_k(n_k)}\na\widetilde{g}_k(n_k)
  = \sqrt{T_k(n_k)}\widetilde{g}_k'(n_k)\na n_k
  \rightharpoonup ng'(n)\na n = n\na g(n)
\end{align*}
weakly in $L^1(\Omega_T)$. Combining this result with \eqref{2.convTV} shows that
\begin{align*}
  \sqrt{T_k(n_k)}\big(\na\widetilde{g}_k(n_k) - \sqrt{T_k(n_k)}
  \na V_k\big) \rightharpoonup n\na g(n)-n\na V
  \quad\mbox{weakly in }L^1(\Omega_T),
\end{align*}
which ends the proof.
\end{proof}

We need the convergence of $(\na\widetilde{g}(n_k))$ and $(\na\widetilde{g}(p_k))$ to pass to the limit in the energy inequality.

\begin{lemma}
It holds that 
\begin{align*}
  (\na\sqrt{n_k},\na\sqrt{p_k})\rightharpoonup(\na\sqrt{n},\na\sqrt{p})
  &\quad\mbox{weakly in }L^2(0,T;L^{2q/(2+q)}(\Omega))^2, \\
  (\na\widetilde{g}_k(n_k),\na\widetilde{g}_k(p_k))
  \rightharpoonup (\sqrt{n}\na g(n),\sqrt{p}\na g(p))
  &\quad\mbox{weakly in }L^1(\Omega_T)^2.
\end{align*}
\end{lemma}

\begin{proof}
It follows from Lemma \ref{lem.Dg} that $\widetilde{g}'_k(s)=\sqrt{s}g'(s)\ge C(s^{-1/2}+s^{1/6})\ge Cs^{-1/2}$ for $s<k$ and $\widetilde{g}'_k(s)=k^{-1/6}s^{1/3}g'(s)\ge CC(s/k)^{1/6}(s^{-5/6}+s^{-1/2})\ge Cs^{-1/2}$ for $s>k$. Then the uniform bound for $\na\widetilde{g}_k(n_k)$ in Lemma \ref{lem.est2} implies directly the bound for $\na\sqrt{n_k}$ in $L^2(0,T;L^{2q/(2+q)}(\Omega))$, showing the first statement. The limit for $\na\sqrt{p_k}$ is shown analogously. Repeating the proof of Lemma \ref{lem.limTnag} with $\sqrt{n_k}$ instead of $\sqrt{T_k(n_k)}$, we obtain the convergence 
\begin{align*}
  \sqrt{n_k}\widetilde{g}'_k(n_k)\to ng'(n)
  \quad\mbox{strongly in }L^2(0,T;L^r(\Omega)).
\end{align*}
We combine this result with the weak convergence of $(\na\sqrt{n_k})$ to conclude that 
\begin{align*}
  \na\widetilde{g}_k(n_k)=2\sqrt{n_k}\widetilde{g}_k'(n_k)\na\sqrt{n_k}
  \rightharpoonup 2ng'(n)\na\sqrt{n} = \sqrt{n}\na g(n)
  \quad\mbox{weakly in }L^1(\Omega_T),
\end{align*}
using that $(2+q)/(2q)+1/r<1$ if we choose $q<8$ and $r<24/7$ sufficiently large. 
\end{proof}


\subsection{Limit $k\to\infty$ in the equation for $D_k$}

We prove the convergence of the terms in the equation for $D_k$. First, we consider  $T_{k/(k+1)}(D_k)^{1/2}\na\widetilde{h}_k(D_k)=T_{k/(k+1)}(D_k)^{1/2}\widetilde{h}'_k(D_k)\na D_k$. We introduce the functions
\begin{align*}
  L(s) &= -\log(1-s) \quad\mbox{for } 0\le s<1, \\
  L_k(s) &= \begin{cases}
  -\log(1-s) &\mbox{for }0\le s\le k/(k+1), \\
  (k+1)s-k+\log(k+1) &\mbox{for }s>k/(k+1).
  \end{cases}
\end{align*}
They satisfy the property $L'_k(D_k)=T_{k/(k+1)}(D_k)^{1/2}\widetilde{h}'_k(D_k)$ for all $0\le s<1$. Moreover, $L_k$ is nondecreasing, $L_k\le L_{k+1}$ on $[0,1)$, and $L_k$ converges to $L$ locally uniformly on $[0,1)$. The aim is to derive uniform bounds for $L_k(D_k)$ and to identify its weak limit.

\begin{lemma}\label{lem.LDH1}
There exists a constant $C>0$ such that 
\begin{align*}
  \|L_k(D_k)\|_{L^2(0,T;H^1(\Omega))}\le C.
\end{align*}
\end{lemma}

\begin{proof}
The gradient bound follow immediately from $L'_k(s)<\widetilde{h}'_k(s)$ for $0<s<1$ and $|\na L_k(D_k)|=L'_k(D_k)|\na D_k|\le \widetilde{h}'_k(D_k)|\na D_k|=|\na\widetilde{h}_k(D_k)|$, showing that $(\na L_k(D_k))$ is bound\-ed in $L^2(\Omega_T)$. By mass conservation and Assumption (A4), we have
\begin{align*}
  \frac{1}{\m(\Omega)}\int_\Omega D_k dx 
  = \frac{1}{\m(\Omega)}\int_\Omega D^I dx = D_\Omega^I < 1.
\end{align*}
Thus, the conditions of Lemma \ref{lem.poincare} in Appendix \ref{sec.poincare} are satisfied, and we infer from the gradient bound that $L_k(D_k)$ is bounded in $L^2(\Omega_T)$. Lemma \ref{lem.poincare} now finishes the proof.
\end{proof}

The uniform bound in Lemma \ref{lem.LDH1} implies the existence of a subsequence such that
\begin{align*}
  L_k(D_k)\rightharpoonup L^* \quad\mbox{weakly in }L^2(\Omega_T)
  \mbox{ as }k\to\infty.
\end{align*}
The next aim is to identify $L^*=L(D)$, where we recall that $D$ is the strong $L^2(\Omega_T)$ limit of $(D_k)$. 

First, we claim that $L(D)\in L^2(\Omega_T)$ and $D<1$ a.e.\ in $\Omega_T$. Indeed, we have $D_\ell\to D$ a.e.\ in $\Omega_T$ as $\ell\to\infty$, up to a subsequence. Hence, since $L_k$ is continuous and $L_k\le L_{k+1}$, for any $k\in\N$,
\begin{align*}
  \int_0^T&\int_\Omega L_k(D)^2 dxdt
  = \int_0^T\int_\Omega \lim_{\ell\to\infty}L_k(D_\ell)^2 dxdt
  = \int_0^T\int_\Omega \liminf_{\ell\to\infty}L_k(D_\ell)^2 dxdt \\
  &\le \int_0^T\int_\Omega \liminf_{\ell\to\infty}L_\ell(D_\ell)^2 dxdt 
  \le \liminf_{\ell\to\infty}\int_0^T\int_\Omega L_\ell(D_\ell)^2 dxdt,
\end{align*}
where the last step follows from Fatou's lemma. The last integral is uniformly bounded. Therefore, $(L_k(D))$ is bounded in $L^2(\Omega_T)$, and again by Fatou's lemma,
\begin{align*}
  \int_0^T\int_\Omega L(D)^2 dxdt \le \liminf_{k\to\infty}
  \int_0^T\int_\Omega L_k(D)^2 dxdt \le C.
\end{align*}
The $L^2(\Omega_T)$-bound for $L(D)=-\log(1-D)$ implies that $D<1$ a.e.

To identify $L^*$ with $L(D)$, we set
\begin{align*}
  D_\eta := (1-\eta)D+\eta, \quad D_{k,\eta} :=(1-\eta)D_k+\eta
  \quad\mbox{for }0<\eta<1.
\end{align*}
Because of the strong convergence of $D_k$, we clearly have $D_{k,\eta}\to D_\eta$ strongly in $L^2(\Omega_T)$ for any fixed $0<\eta<1$. The proof of Lemma \ref{lem.LDH1} shows that $(L_k(D_{k,\eta}))$ is bounded in $L^2(0,T;H^1(\Omega))$, and the previous arguments imply that $L(D_\eta)\in L^2(\Omega_T)$.
Furthermore, let $\psi:\Omega_T\to[0,1]$ be a smooth function and set
\begin{equation*}
	D_{\eta,\psi} := (1 - \eta \psi)D + \eta \psi = D + (1-D)\eta \psi, \quad \mbox{for } 0 < \eta < 1.
\end{equation*}
Since $D_{\eta,\psi} \leq D_\eta$, the monotonicity of $L$ implies that $L(D_{\eta,\psi}) \in L^2(\Omega_T)$. With these preparations, we can prove that $L^*=L(D)$.
\begin{lemma}\label{lem.Lstar}
It holds that $L^*=L(D)$ and
\begin{align*}
  L_k(D_k)\rightharpoonup L(D)\quad\mbox{weakly in }L^2(0,T;H^1(\Omega)).
\end{align*}
\end{lemma}

\begin{proof}
The proof is based on the monotonicity of $L_k$ (Minty trick). We pass to the limit $k\to\infty$ in
\begin{align*}
  0 \le \int_0^T\int_\Omega
  (L_k(D_{\eta,\psi})-L_k(D_k))(D_{\eta,\psi}-D_k)dxdt,
\end{align*}
leading to
\begin{align*}
  0 \le \frac{1}{\eta}\int_0^T\int_\Omega(L(D_{\eta,\psi})-L^*)
  (D_{\eta,\psi}-D)dxdt
  = \int_0^T\int_\Omega(L(D_{\eta,\psi})-L^*)(1-D)\psi dxdt.
\end{align*}
By dominated convergence, the limit $\eta\to 0$ gives
\begin{align}\label{2.Lstar}
  0 \le \int_0^T\int_\Omega(L(D)-L^*)(1-D)\psi dxdt.
\end{align}

Next, we show the inverse inequality. The monotonicity $L_k\le L_{k+1}$, the weak convergence of $(L_k(D_k))$, and Fatou's lemma yield the following inequalities:
\begin{align*}
  \int_0^T&\int_\Omega(1-D)\psi L_k(D)dxdt
  = \int_0^T\int_\Omega(1-D)\psi \lim_{\ell\to\infty}L_k(D_\ell)dxdt \\
  &\le \int_0^T\int_\Omega(1-D)\psi \liminf_{\ell\to\infty}L_\ell(D_\ell)dxdt
  \le \liminf_{\ell\to\infty}\int_0^T\int_\Omega(1-D)\psi L_\ell(D_\ell)dxdt \\
  &\le \int_0^T\int_\Omega(1-D)\psi L^*dxdt.
\end{align*}
It follows from dominated convergence that
\begin{align*}
  \int_0^T\int_\Omega(1-D)\psi L(D)dxdt 
  \le \int_0^T\int_\Omega(1-D)\psi L^*dxdt,
\end{align*}
or equivalently,
\begin{align*}
  \int_0^T\int_\Omega(L(D)-L^*)(1-D)\psi dxdt \le 0.
\end{align*}
Together with \eqref{2.Lstar} this shows that
\begin{align*}
	\int_0^T\int_\Omega(L(D)-L^*)(1-D)\psi dxdt = 0.
\end{align*}
Since $\psi$ is arbitrary, we find that $(L(D)-L^*)(1-D)=0$ a.e. It follows from $D<1$ a.e.\ that $L(D)=L^*$ a.e. Hence, by Lemma \ref{lem.LDH1}, $ L_k(D_k) \rightharpoonup L(D)$ weakly in $L^2(0,T;H^1(\Omega))$.
\end{proof}

The previous lemma implies that
\begin{align}\label{2.convh}
  &T_{k/(k+1)}(D_k)^{1/2}\na\widetilde{h}_k(D_k)
  = T_{k/(k+1)}(D_k)^{1/2}\widetilde{h}'_k(D_k)\na D_k
  = L_k'(D_k)\na D_k \\
  &\phantom{xx}= \na L_k(D_k)\rightharpoonup \na L(D) = -\na\log(1-D)
  = D\na h(D) \quad\mbox{weakly in }L^2(\Omega_T). \nonumber
\end{align}
The next step is the convergence of $(T_{k/(k+1)}(D_k)\na V_k)$ and the identification of the limit. 

Set $T_1(s)=\max\{0,\min\{1,s\}\}$. We deduce from
\begin{align*}
  |T_{k/(k+1)}(D_k)-T_1(D)| &= |T_{k/(k+1)}(D_k)-T_{k/(k+1)}(D)|
  + |T_{k/(k+1)}(D)-T_1(D)| \\
  &\le |D_k-D| + \frac{1}{k+1}
\end{align*}
and the strong convergence of $(D_k)$ that
\begin{align*}
  \int_0^T\int_\Omega|T_{k/(k+1)}(D_k)-T_1(D)|^2 dxdt
  \le 2\int_0^T\int_\Omega|D_k-D|^2 dxdt + \frac{C}{(k+1)^2}\to 0.
\end{align*}
The property $D<1$ a.e.\ implies that $T_1(D)=D$ and thus
$T_{k/(k+1)}(D_k)\to D$ strongly in $L^2(\Omega_T)$. We deduce from the convergence of $\na V_k\rightharpoonup\na V$ weakly in $L^2(\Omega_T)$ that $T_{k/(k+1)}(D_k)\na V_k\rightharpoonup D\na V$ weakly in $L^1(\Omega_T)$. We know from Lemma \ref{lem.est1} that $T_{k/(k+1)}(D_k)\na V_k$ $=T_{k/(k+1)}(D_k)^{1/2}\cdot T_{k/(k+1)}(D_k)^{1/2}\na V_K)$ is uniformly bounded in $L^\infty(0,T;L^2(\Omega))$. This yields the convergence
\begin{align*}
  T_{k/(k+1)}(D_k)\na V_k \rightharpoonup D\na V
  \quad\mbox{weakly in }L^2(\Omega_T).
\end{align*}
It follows from \eqref{2.convh} that
\begin{align*}
  T_{k/(k+1)}(D_k)^{1/2}\big(\na\widetilde{h}_k(D_k)
  + T_{k/(k+1)}(D_k)^{1/2}\na V_K\big)\rightharpoonup D\na h(D)+D\na V
\end{align*}
weakly in $L^2(\Omega_T)$.

Performing the limit $k\to\infty$ in the weak formulation \eqref{2.nk}--\eqref{2.Vk}, using formulation \eqref{2.fluxk} of the fluxes, leads to the weak formulation of \eqref{1.n}--\eqref{1.V}. The limit $(n,p,V)$ satisfies the Dirichlet boundary conditions \eqref{1.Dbc}  since $n_k=\bar{n}$, $p_k=\bar{p}$, and $V_k=\bar{V}$ on $\Gamma_D$. Finally, the limit $n_k\rightharpoonup n$ weakly in $W^{1,7/5}(0,T;W_D^{1,2q/(q+4)}(\Omega)')\hookrightarrow C^0([0,T];W_D^{1,2q/(q+4)}(\Omega)')$ and the property $n_k(\cdot,0)=n^I$ in the sense of $W_D^{1,2q/(q+4)}(\Omega)'$ show that $n$ satisfies the initial condition in \eqref{1.ic} in a weak sense. Similarly, $p$ and $D$ also satisfy the initial conditions. 


\subsection{Energy inequality}

We identify the weak limit of $\na\widetilde{h}_k(D_k)$, which is needed in the limit of the energy inequality \eqref{2.Ekineq}, leading to \eqref{1.Eineq}.

\begin{lemma}
Recall that $\widetilde{h}(s)=2\tanh^{-1}\sqrt{s}-2\tanh^{-1}\sqrt{1/2}$. It holds that
\begin{align*}
  \na\widetilde{h}_k(D_k)\rightharpoonup \na\widetilde{h}(D)
  \quad\mbox{weakly in }L^2(\Omega_T).
\end{align*}
\end{lemma}

\begin{proof}
The function $\widetilde{h}_k$ is written explicitly as
\begin{align*}
  \widetilde{h}_k(s) = \begin{cases}
  2\tanh^{-1}\sqrt{s}-2\tanh^{-1}\sqrt{1/2} 
  &\mbox{for }0\le s\le\frac{k}{k+1}, \\
  \sqrt{\frac{k+1}{k}}((k+1)s-k) 
  + 2\tanh^{-1}\sqrt{\frac{k}{k+1}}
  &\mbox{for }s>\frac{k}{k+1}.
  \end{cases}
\end{align*}
By Lemma \ref{lem.est1}, the sequence $(\na\widetilde{h}_k(D_k))$ is bounded in $L^2(\Omega_T)$. Proceeding as in the proof of Lemma \ref{lem.LDH1}, we can show that $(\widetilde{h}_k(D_k))$ is bounded in $L^2(0,T;H^1(\Omega))$. We deduce from
\begin{align*}
  |\na\widetilde{h}_k(D_{k,\eta})| \le
  \begin{cases}
  |\na\widetilde{h}_k(D_k)| &\mbox{for }D_{k,\eta}\le k/(k+1), \\
  \sqrt{2}|\na L_k(D_{k,\eta})| &\mbox{for }D_{k,\eta}>k/(k+1)
  \end{cases}
\end{align*}
a uniform bound for $\na\widetilde{h}_k(D_{k,\eta})$ in $L^2(\Omega_T)$. 
The bound for $(\widetilde{h}_k(D_k))$ implies a bound for $(\widetilde{h}_k(D_{k,\eta}))$ in $L^2(\Omega_T)$, and we have $\widetilde{h}(D_\eta)\in L^2(\Omega_T)$. Finally, the proof of Lemma \ref{lem.Lstar} shows that
\begin{align*}
  \widetilde{h}_k(D_k)\rightharpoonup 2\tanh^{-1}\sqrt{D}
  - 2\tanh^{-1}\sqrt{1/2}
  &\quad\mbox{weakly in }L^2(\Omega_T), \\
  \na\widetilde{h}_k(D_k)\rightharpoonup 2\na\tanh^{-1}\sqrt{D}
  &\quad\mbox{weakly in }L^2(\Omega_T),
\end{align*}
ending the proof.
\end{proof}


\section{Proof of Theorem \ref{thm.bound}}\label{sec.bound}

Since the proof is similar to that one of \cite[Theorem 2]{JuVe23}, we present only the key ideas, proceeding formally. First, we notice that Assumption (A5) with $r=3+\eps>3$ yields
\begin{align}\label{3.V}
  \|V\|_{L^\infty(0,T;W^{1,r}(\Omega))} 
  \le C\|n-p-D+A\|_{L^\infty(0,T;L^{3r/(r+3)}(\Omega))} + C
  \le C,
\end{align}
by Lemma \ref{lem.est1}, since $3r/(r+3)\le 5/3$ if $\eps>0$ is sufficiently small. To simplify, we assume that $\bar{n}=0$. The proof can be extended in a straightforward way to general boundary data $\bar{n}$, but at the price of more elaborate technical estimations (see, e.g., \cite[Section 3]{JuVe23}). We wish to use $n^q$ for $q\in\N$ as a test function in the weak formulation of \eqref{1.n}. To justify this step, we need to show that $n\in L^\infty(0,T;L^q(\Omega))$. This property is proved in \cite[Sec.~3]{JuVe23} in a similar context and therefore, we do not present the quite technical proof here. With this test function, we obtain
\begin{align*}
  \frac{1}{q+1}\frac{d}{dt}\int_\Omega n^{q+1}dx
  + q\int_\Omega ng'(n)n^{q-1}|\na n|^2 dx
  = q\int_\Omega n^q\na V\cdot\na n dx.
\end{align*}
We deduce from Lemma \ref{lem.Dg} that $ng'(n)\ge c(1+n^{2/3})$ and hence, using H\"older's inequality for the drift term, integrating over $(0,t)$, and possibly lowering the constant $c>0$,
\begin{align}\label{3.nq1}
  \frac{1}{q+1}&\int_\Omega(n(t)^{q+1}-n(0)^{q+1})dx
  + \frac{c}{q+1}\int_0^t\int_\Omega\big(|\na n^{(q+1)/2}|^2
  + |\na n^{(3q+5)/6}|^2\big)dxds \\
  &\le \frac{6q}{3q+5}\int_\Omega n^{(3q+1)/6}\na V\cdot\na  
  n^{(3q+5)/6} dx \nonumber \\
  &\le C\int_0^t\|n^{(3q+1)/6}\|_{L^{(6+2\eps)/(1+\eps)}(\Omega)}
  \|\na V\|_{L^{3+\eps}(\Omega)}
  \|\na n^{(3q+5)/6}\|_{L^2(\Omega)}ds \nonumber \\
  &\le C\int_0^t\|n^{(3q+1)/6}\|_{L^{(6+2\eps)/(1+\eps)}(\Omega)}
  \|\na n^{(3q+5)/6}\|_{L^2(\Omega)}ds, \nonumber 
\end{align}
where in the last step we have taken into account \eqref{3.V}. We apply the Gagliardo--Nirenberg inequality with $\theta=(15+3\eps)/(15+5\eps)<1$ (at this point we need the regularity in $W^{1,r}(\Omega)$ with $r=3+\eps>3$; $\eps=0$ would not be sufficient) and Young's inequality:
\begin{align*}
  \|n^{(3q+1)/6}\|_{L^{(6+2\eps)/(1+\eps)}(\Omega)}
  &= \|n^{(3q+5)/6}\|_{L^{(3q+1)(6+2\eps)/((3q+5)(1+\eps))}
  (\Omega)}^{(3q+1)/(3q+5)} \\
  &\le C + C\|n^{(3q+5)/6}\|_{L^{(3q+1)(6+2\eps)/((3q+5)(1+\eps))}
  (\Omega)} \\
  &\le C + C\|\na n^{(3q+5)/6}\|_{L^2(\Omega)}^{\theta}
  \|n^{(3q+5)/6}\|_{L^1(\Omega)}^{1-\theta}.
\end{align*}
We insert this expression into \eqref{3.nq1}, multiply by $q+1$, and use the Young inequality $(q+1)a^{1+\theta}b^{1-\theta}\le a^2 + C_\theta(q+1)^{2/(1-\theta)}b^2$, with $C_\theta>0$ depending solely on $\theta$:
\begin{align*}
  \|n(t)&\|_{L^{q+1}(\Omega)}^{q+1}
  + c\int_0^t\|\na n^{(3q+5)/6}\|_{L^2(\Omega)}^2 ds \\
  &\le \|n^I\|_{L^\infty(\Omega)}^{q+1}
  + C(q+1)\int_0^t\big(1+\|\na n^{(3q+5)/6}\|_{L^2(\Omega)}^{1+\theta}
  \|n^{(3q+5)/6}\|_{L^1(\Omega)}^{1-\theta}\big)ds \\
  &\le \|n^I\|_{L^\infty(\Omega)}^{q+1} + CT(q+1)
  + \frac{c}{2}\int_0^t\|\na n^{(3q+5)/6}\|_{L^2(\Omega)}^2 ds \\
  &\phantom{xx}
  + C(q+1)^{2/(1-\theta)}\int_0^t\|n^{(3q+5)/6}\|_{L^1(\Omega)}^2 ds.
\end{align*}
The third term on the right-hand side can be absorbed by the left-hand side. Then, taking the supremum over $(0,T)$ and observing that $\|n^{(3q+5)/6}\|_{L^1(\Omega)}^2 = \|n\|_{L^{(3q+5)/6}(\Omega)}^{(3q+5)/3}$:
\begin{align}\label{3.q}
  \|n\|_{L^\infty(0,T;L^{q+1}(\Omega))}^{q+1}
  \le \|n^I\|_{L^\infty(\Omega)}^{q+1} + C(q+1)
  + C(q+1)^{2/(1-\theta)}\|n\|_{L^\infty(0,T;L^{(3q+5)/6}(\Omega))
  }^{(3q+5)/3}.
\end{align}
We set $q_k:=q+1$ and $q_{k-1}:=(3q+5)/6$, which defines the recursion $q_{k-1}=(3q_k+2)/6$ or $q_k=2q_{k-1}-2/3$ with the explicit solution
\begin{align*}
  q_k = 2^k\frac{3q_0-2}{3} + \frac23, \quad k\in\N,
\end{align*}
where the initialization $q_0>2/3$ is arbitrary. Setting 
\begin{align*}
  b_k = \|n\|_{L^\infty(0,T;L^{q_k}(\Omega))}^{q_k}
  + \|n^I\|_{L^\infty(\Omega)}^{q_k} + 1,
\end{align*}
an elementary computation shows that \eqref{3.q} can be written as
\begin{align*}
  b_k \le C^k q_k^{2/(1-\theta)} b_{k-1}^2, \quad k\in\N.
\end{align*}
This inequality can be solved as in \cite[Sec.~3]{JuVe23}:
\begin{align*}
  \|n\|_{L^\infty(0,T;L^{q_k}(\Omega))}
  \le (C3^{2/(1-\theta)})^{(2^{k+1}-k-2)/q_k}\big(
  \|n\|_{L^\infty(0,T;L^{q_0}(\Omega))}^{q_0} 
  + \|n^I\|_{L^\infty(\Omega)}^{q_0}
  + 1\big)^{2^k/q_k}.
\end{align*}
It is shown in \cite[Sec.~3]{JuVe23} that the exponents on the right-hand side are bounded uniformly in $k$. Choosing $q_0\le 5/3$, it follows from Lemma \ref{lem.est1} that $N:=\|n\|_{L^\infty(0,T;L^{q_0}(\Omega))}$  is bounded. Then the limit $k\to\infty$ shows that $\|n\|_{L^\infty(0,T;L^\infty(\Omega))}\le C(n^I,N)$.

\begin{appendix}
\section{Fermi--Dirac integrals}\label{sec.app}

The Fermi--Dirac integral of order $j>-1$ is defined by 
\begin{align*}
  \F_j(z) = \frac{1}{\Gamma(j+1)}\int_0^\infty\frac{s^j}{1+e^{s-z}}ds,
  \quad z\in\R,
\end{align*}
where $\Gamma(j+1)=\int_0^\infty s^j e^{-s}ds$ is the Gamma function. This function has the property $\F_j'=\F_{j-1}$ for $j>0$. In the following, we write $A\sim B$ if there exist constants $C_1,C_2>0$ such that $A\le C_1B\le C_2 A$.

\begin{lemma}\label{lem.g1}
It holds for any $j>-1$ and $z\in\R$ that
\begin{align*}
  \F_j(z) \sim e^z\mathrm{1}_{\{z\le 0\}} 
  + (z^{j+1}+1)\mathrm{1}_{\{z>0\}}.
\end{align*}
\end{lemma}

\begin{proof}
Let first $z\le 0$. Then $e^{s-z} < 1+e^{s-z}\le 2e^{s-z}$ and  
\begin{align*}
  \F_j(z) \ge\frac{1}{\Gamma(j+1)}\int_0^\infty\frac{s^j}{2e^{s-z}}ds
  =\frac{e^z}{2}, \quad
  \F_j(z) \le\frac{1}{\Gamma(j+1)}\int_0^\infty\frac{s^j}{e^{s-z}}ds
  = e ^z.
\end{align*}
This shows that $\F_j(z)\sim e^z$ for $z\le 0$. For $z>0$, we write
$\F_j(z)=I_1+I_2$, where
\begin{align*}
  I_1 = \frac{1}{\Gamma(j+1)}\int_0^z\frac{s^j}{1+e^{s-z}}ds, \quad
  I_2 = \frac{1}{\Gamma(j+1)}\int_z^\infty\frac{s^j}{1+e^{s-z}}ds.
\end{align*}
We infer from $s\le z$ and hence $e^{s-z}\le 1$ that
\begin{align*}
  I_1 &\ge \frac{1}{2\Gamma(j+1)}\int_0^z s^j ds
  = \frac{z^{j+1}}{2(j+1)\Gamma(j+1)} = \frac{z^{j+1}}{2\Gamma(j+2)}, \\
  I_1 &\le \frac{1}{\Gamma(j+1)}\int_0^z s^j ds
  = \frac{z^{j+1}}{\Gamma(j+2)},
\end{align*}
To estimate $I_2$, we assume first that $j\ge 0$ and substitute $y=s-z\ge 0$:
\begin{align*}
  I_2 &= \frac{1}{\Gamma(j+1)}\int_0^\infty
  \frac{(y+z)^{j}}{1+e^y}dy
  \ge \frac{1}{\Gamma(j+1)}\int_0^\infty\frac{y^j}{2e^y}dy = \frac12, \\
  I_2 &\le \frac{1}{\Gamma(j+1)}\bigg(\int_0^z+\int_z^\infty\bigg)
  \frac{(y+z)^{j}}{1+e^y}dy \\
  &\le \frac{1}{\Gamma(j+1)}\int_0^z\frac{(2z)^j}{1+e^y}dy
  + \frac{1}{\Gamma(j+1)}\int_z^\infty\frac{(2y)^j}{1+e^y}dy \\
  &\le \frac{(2z)^j}{\Gamma(j+1)}\int_0^z e^{-y}dy
  + \frac{2^j}{\Gamma(j+1)}\int_z^\infty\frac{y^j}{e^y}dy
  \le \frac{(2z)^j}{\Gamma(j+1)} + 2^j.
\end{align*}
We conclude that 
\begin{align*}
  \F_j(z) \le \frac{z^{j+1}}{\Gamma(j+2)} 
  + \frac{(2z)^j}{\Gamma(j+1)} + 2^j, \quad
  \F_j(z) \ge \frac{z^{j+1}}{2\Gamma(j+2)} + \frac12, 
\end{align*}
proving that $\F_j(z)\sim z^{j+1}+1$ for $z>0$ and $j\ge 0$. 

Finally, let $-1<j<0$. Then, arguing as before, $I_1\le z^{j+1}/\Gamma(j+2)$,
\begin{align*}
  I_2 = \frac{1}{\Gamma(j+1)}\int_0^\infty
  \frac{(y+z)^{j}}{1+e^y}dy 
  \le \frac{1}{\Gamma(j+1)}\int_0^\infty\frac{y^{j}}{e^y}dy = 1,
\end{align*}
and hence $\F_j(z)\le z^{j+1}/\Gamma(j+2)+1$. The estimate from below requires a more careful computation. As the function $\F_j$ is continuous and strictly increasing, its minimum on $[0,1]$ equals $\F_j(0)$. Thus, for $z\in[0,1]$,
\begin{align*}
  \F_j(z) \ge \F_j(0) = \frac{1}{\Gamma(j+1)}\int_0^\infty  \frac{s^j}{1+e^s}ds > \frac{1}{\Gamma(j+1)}\int_0^\infty  \frac{s^j}{2e^s}ds = \frac12.
\end{align*}
Let $z>1$. We always have $I_2\ge 0$. The inequality $s^j\ge 1\ge z^{-1}$ for $1\le s\le z$ implies that $s^j\ge \frac12(s^j+z^{-1})$ and hence,
\begin{align*}
  I_1 &= \frac{1}{\Gamma(j+1)}\bigg(\int_0^1+\int_1^z\bigg)
  \frac{s^j}{1+e^{s-z}}ds \\
  &\ge \frac{1}{2\Gamma(j+1)}\int_0^1 s^j ds
  + \frac{1}{2\Gamma(j+1)}\int_1^z\frac12(s^j + z^{-1})ds \\
  &= \frac{1}{4\Gamma(j+1)}\int_0^z s^j ds 
  + \frac{1}{4\Gamma(j+1)}\int_0^1 s^j ds 
  + \frac{1}{4\Gamma(j+1)}\int_1^z z^{-1}ds \\
  &= \frac{z^{j+1}}{4\Gamma(j+2)} + \frac{1}{4\Gamma(j+1)}
  \bigg(\frac{1}{j+1} + 1 - \frac{1}{z}\bigg)
  \ge \frac{z^{j+1}}{4\Gamma(j+2)} + \frac{1}{4\Gamma(j+2)}.
\end{align*}
Combining the estimates for $z\in[0,1]$ and $z>1$ yields $\F_j(z)\ge C(z^{j+1}+1)$ for all $z\ge 0$. This is the desired lower bound, which concludes the proof.
\end{proof}

\begin{corollary}\label{coro.g}
Let $j>0$. Then for $z\in\R$,
\begin{align*}
  \F_j'(z) = \F_{j-1}(z)\sim \F_j(z)\mathrm{1}_{\{z\le 0\}}
  + \F_j(z)^{j/(j+1)}\mathrm{1}_{\{z>0\}}.
\end{align*}
\end{corollary}

\begin{proof}
If $z\le 0$ then, by Lemma \ref{lem.g1}, $\F_{j-1}(z)\sim e^z \sim \F_j(z)$. Furthermore, by the same lemma, if $z>0$, we have $\F_{j-1}(z)\sim z^j = (z^{j+1})^{j/(j+1)} \sim \F_j(z)^{j/(j+1)}$.
\end{proof}

\begin{lemma}\label{lem.Dg}
Recall that $g=\F_{1/2}^{-1}$. It holds for $z>0$ that
\begin{align*}
  g'(z)\sim z^{-1} + z^{-1/3}.
\end{align*}
\end{lemma}

\begin{proof}
Let $z>0$ and $y=\F_{1/2}^{-1}(z)$. Then, by Corollary \ref{coro.g},
\begin{align*}
  g'(z) &= \frac{1}{\F'_{1/2}(y)} = \frac{1}{\F_{-1/2}(y)} \sim 
  \frac{1}{\F_{1/2}(y)\mathrm{1}_{\{y\le 0\}}
  + \F_{1/2}(y)^{1/3}\mathrm{1}_{\{y>0\}}} \\
  &= \frac{1}{z\mathrm{1}_{\{y\le 0\}} + z^{1/3}\mathrm{1}_{\{y>0\}}}
  = z^{-1}\mathrm{1}_{\{y\le 0\}} + z^{-1/3}\mathrm{1}_{\{y>0\}},
\end{align*}
which shows that $g'(z)\le C(z^{-1}+z^{-1/3})$ for $z>0$. For the lower bound, we distinguish the cases $0<z\le\F_{1/2}(0)$ (or $y\le 0$) and $z>\F_{1/2}(0)$ (or $y>0$). The case $z\le\F_{1/2}(0)$ implies that $z<1$ (since $1/2<\F_{1/2}(0)<1$). Then the inequality $z^{-1}>z^{-1/3}$ yields $z^{-1}\mathrm{1}_{\{y\le 0\}}>(z^{-1}+z^{-1/3})/2$ and $g'(z)\le C(z^{-1}+z^{-1/3})$.  

Let $z>\F_{1/2}(0)$. If $z\ge 1$, we have $z^{-1/3}\mathrm{1}_{\{y>0\}}\ge z^{-1}\mathrm{1}_{\{y>0\}}$ and $g'(z)\ge C(z^{-1}+z^{-1/3})$. If $\F_{1/2}(0)<z<1$, we observe that $\F_{1/2}(0)<z$ yields $\F_{1/2}(0)z^{-1}<1<z^{-1/3}$, showing that
\begin{align*}
  z^{-1/3}\mathrm{1}_{\{y>0\}} 
  > \frac{\F_{1/2}(0)}{2}(z^{-1}+z^{-1/3})\mathrm{1}_{\{y>0\}}
\end{align*}
and again $g'(z)\ge C(z^{-1}+z^{-1/3})$. Summarizing these estimates, we conclude the proof.
\end{proof}

Our final result is used in the proof of Lemma \ref{lem.limTnag}. 

\begin{lemma}\label{lem.DzDg}
There exists a constant $C>0$ such that for $z>0$,
\begin{align*}
  \frac{d}{dz}(zg'(z)) \le C\big(\mathrm{1}_{\{z<\F_{1/2}(0)\}}
  + z^{-1/3}\mathrm{1}_{\{z\ge \F_{1/2}(0)\}}\big).
\end{align*}
\end{lemma}

\begin{proof}
Let $z<\F_{1/2}(0)$ and set $y=\F_{1/2}^{-1}(z)<0$. Then
\begin{align}\label{2.dzg}
  \frac{d}{dz}(zg'(z)) = g'(z) + zg''(z)
  = \frac{\F'_{1/2}(y)^2 - \F_{1/2}(y)\F_{1/2}''(y)}{\F'_{1/2}(y)^3}.
\end{align}
Let $N(y):=\F'_{1/2}(y)^2 - \F_{1/2}(y)\F_{1/2}''(y)$. We compute the derivatives 
\begin{align*}
  \F'_{1/2}(z) = \frac{2}{\sqrt{\pi}}\int_0^\infty
  \frac{\sqrt{s} e^{s-z}}{(1+e^{s-z})^2}ds, \quad
  \F''_{1/2}(z) = \frac{2}{\sqrt{\pi}}\int_0^\infty
  \frac{\sqrt{s}e^{s-z}(e^{s-z}-1)}{(1+e^{s-z})^3}ds,
\end{align*}
yielding
\begin{align*}
  N(y) &= \frac{4}{\pi}\int_0^\infty\int_0^\infty
  \frac{\sqrt{st}e^{s-y}e^{t-y}dsdt}{(1+e^{s-y})^2(1+e^{t-y})^2}
  - \frac{4}{\pi}\int_0^\infty\int_0^\infty
  \frac{\sqrt{st}e^{s-y}(e^{s-y}-1)dsdt}{(1+e^{s-y})^3(1+e^{t-y})} \\
  &= \frac{4}{\pi}\int_0^\infty\int_0^\infty\sqrt{st} e^{s-y}
  \frac{2e^{t-y}-e^{s-y}+1}{(1+e^{s-y})^3(1+e^{t-y})^2}dsdt \\
  &= \frac{4}{\pi}\int_0^\infty\int_0^\infty\sqrt{st}e^{s-y}
  \bigg(\frac{2}{(1+e^{s-y})^3(1+e^{t-y})}
  - \frac{1}{(1+e^{s-y})^2(1+e^{t-y})^2}\bigg)dsdt \\
  &= \frac{2}{\sqrt{\pi}}\F_{1/2}(y)\int_0^\infty 
  \frac{2\sqrt{s} e^{s-y}}{(1+e^{s-y})^3}ds
  - \frac{2}{\sqrt{\pi}}\F_{1/2}'(y)\int_0^\infty
  \frac{\sqrt{t}}{(1+e^{t-y})^2}dt \\
  &\le \frac{2}{\sqrt{\pi}}\F_{1/2}(y)\int_0^\infty 
  \frac{2\sqrt{s} e^{s-y}}{(1+e^{s-y})^3}ds.
\end{align*}
We estimate the remaining integral:
\begin{align*}
  \frac{2}{\sqrt{\pi}}\int_0^\infty
  \frac{\sqrt{s} e^{s-y}}{(1+e^{s-y})^3}ds
  \le \frac{2}{\sqrt{\pi}}e^{y}\int_0^\infty
  \frac{\sqrt{s} e^{s-y}}{(1+e^{s-y})^2}ds
  = e^y\F'_{1/2}(y).
\end{align*}
Therefore, using Lemma \ref{lem.g1} for $y=\F_{1/2}^{-1}(z)<0$,
\begin{align*}
  N(y) \le 2e^y\F_{1/2}(y)\F'_{1/2}(y) \le C_1e^{3y}, \quad
  \F'_{1/2}(y)^3 = \F_{-1/2}(y)^3\ge C_2e^{3y}.
\end{align*}
We conclude from \eqref{2.dzg} for $z<\F_{1/2}(0)$ that
\begin{align*}
  \frac{d}{dz}(zg'(z)) = \frac{N(y)}{\F_{1/2}'(y)^3} 
  \le \frac{C_1}{C_2}.  
\end{align*}

Next, let $z\ge \F_{1/2}(0)$ (or $y\ge 0$). We know from the proof of Lemma \ref{lem.Dg} that in this case $g'(z)\le Cz^{-1/3}$. According to \eqref{2.dzg}, it remains to show that
\begin{align*}
  zg''(z) = -\frac{\F_{1/2}(y)\F_{1/2}''(y)}{\F_{1/2}'(y)^3}
  \le Cz^{-1/3}\quad\mbox{for }z\ge \F_{1/2}(0).
\end{align*}
To this end, we fix some $y_0>\F_{1/2}(0)$ and choose $0<\eps<y_0$. For $y>y_0$, we split the integral $-\F''_{1/2}$ into two parts, $-\F''_{1/2}(y)=I_3+I_4$, where
\begin{align*}
  I_3 = -\frac{2}{\sqrt{\pi}}\int_0^\eps
  \frac{\sqrt{s}e^{s-y}(e^{s-y}-1)}{(1+e^{s-y})^3}ds, \quad 
  I_4 = -\frac{2}{\sqrt{\pi}}\int_\eps^\infty
  \frac{\sqrt{s}e^{s-y}(e^{s-y}-1)}{(1+e^{s-y})^3}ds.
\end{align*}
The estimate of $I_3$ is straightforward:
\begin{align*}
  I_3 \le \frac{2}{\sqrt{\pi}}e^{\eps-y}\int_0^\eps\sqrt{s}ds
  = \frac{4}{3\sqrt{\pi}}e^{\eps-y}\eps^{3/2}.
\end{align*}
We integrate by parts in $I_4$ twice and split the resulting integral into two parts:
\begin{align*}
  I_4 &= \frac{2}{\sqrt{\pi}}\bigg(
  \frac{\sqrt{s}e^{s-y}}{(1+e^{s-y})^2}\bigg|_\eps^\infty
  + \int_\eps^\infty\frac{s^{-1/2} e^{s-y}}{2(1+e^{s-y})^2}ds\bigg) \\
  &= \frac{2}{\sqrt{\pi}}\bigg\{
  -\frac{\sqrt{\eps}e^{\eps-y}}{(1+e^{\eps-y})^2}
  + \bigg(\frac{s^{-1/2}}{2(1+e^{s-y})}\bigg|_\eps^\infty
  + \int_\eps^\infty\frac{s^{-3/2}}{4(1+e^{s-y})}ds\bigg)\bigg\} \\
  &= -\frac{2}{\sqrt{\pi}}\bigg(
  \frac{\sqrt{\eps}e^{\eps-y}}{(1+e^{\eps-y})^2}
  + \frac{\eps^{-1/2}}{2(1+e^{s-y})}\bigg)
  + \frac{2}{\sqrt{\pi}}(I_5 + I_6),
\end{align*}
where
\begin{align*}
  I_5 &= \int_\eps^y\frac{s^{-3/2}}{4(1+e^{s-y})}ds
  \le \frac{1}{4(1+e^{\eps-y})}\int_\eps^y s^{-3/2}ds
  = \frac{\eps^{-1/2}-y^{-1/2}}{2(1+e^{\eps-y})}, \\
  I_6 &= \int_y^\infty\frac{s^{-3/2}}{4(1+e^{s-y})}ds
  \le \frac12\int_y^\infty s^{-3/2}ds = y^{-1/2}.
\end{align*}
Summarizing these estimates, we observe that the contributions that are singular for $\eps\to 0$ cancel, and we end up with
\begin{align*}
  -\F_{1/2}''(y) &\le\frac{2}{\sqrt{\pi}}\bigg\{
  \frac23 e^{\eps-y}\eps^{3/2} 
  - \bigg(\frac{\sqrt{\eps}e^{\eps-y}}{(1+e^{\eps-y})^2}
  + \frac{\eps^{-1/2}}{2(1+e^{s-y})}\bigg)
  + \frac{\eps^{-1/2}-y^{-1/2}}{2(1+e^{\eps-y})}
  + y^{-1/2}\bigg\} \\
  &= \frac{2}{\sqrt{\pi}}\bigg(
  \frac23 e^{\eps-y}\eps^{3/2}
  - \frac{\sqrt{\eps}e^{\eps-y}}{(1+e^{\eps-y})^2}
  + \frac{1+2e^{\eps-y}}{2(1+e^{\eps-y})}y^{-1/2}\bigg).
\end{align*}
Since $\eps>0$ is arbitrary, we can pass to the limit $\eps\to 0$ to find that for all $y\ge 0$,
\begin{align*}
  -\F''_{1/2}(y) \le \frac{2}{\sqrt{\pi}}
  \frac{1+2e^{-y}}{2(1+e^{-y})}y^{-1/2} \le C y^{-1/2}.
\end{align*}
We use the previous inequality to estimate the nominator and Corollary \ref{coro.g} to estimate the denominator in \eqref{2.dzg}:
\begin{align*}
  zg''(z) = -\frac{\F_{1/2}(y)\F''_{1/2}(y)}{\F_{1/2}'(y)^3}
  \le C\frac{\F_{1/2}(y)y^{-1/2}}{\F_{1/2}(y)} 
  \le \frac{C}{y^{1/2}} \quad\mbox{for }y>y_0.
\end{align*}
For $0\le y\le y_0$, the expression $zg''(z)$ is bounded since $\F_{1/2}$ and its derivatives are bounded in $[0,y_0]$. In case $y>y_0$, we wish to have an upper bound in terms of $z$. For this, we notice that, by Lemma \ref{lem.g1}, $z=\F_{1/2}(y)\le C_1(y^{3/2}+1)$. It follows from $y>y_0$ that $z\le C_1(1+y_0^{-3/2})y^{3/2}=:C_2 y^{3/2}$ and 
\begin{align*}
  zg''(z) \le Cy^{-1/2} \le CC_2^{1/3}z^{-1/3}.
\end{align*}
This concludes the proof.
\end{proof}


\section{A nonlinear Poincar\'e--Wirtinger-type lemma}\label{sec.poincare}

We show a nonlinear version of the Poincar\'e--Wirtinger inequality.

\begin{lemma}\label{lem.poincare}
Let $\Omega\subset\R^d$ ($d\ge 1$) be a bounded domain, let $f:[0,b)\to[0,\infty)$ with $b\in\R\cup\{+\infty\}$, be a strictly increasing function, let $u\in L^1(\Omega)$ satisfy $f(u)\in H^1(\Omega)$ and $u_\Omega:=\m(\Omega)^{-1}\int_\Omega udx<b$. Then for any $\widehat{u}\in(u_\Omega,b)$ one has
\begin{align*}
  \|f(u)\|_{L^2(\Omega)}^2\le 2m(\Omega)f(\widehat{u})^2 
  + 4C_P\bigg(1+\frac{\widehat{u}}{\widehat{u}-u_\Omega}\bigg)
  \|\na f(u)\|_{L^2(\Omega)}^2,
\end{align*}
where $C_P>0$ is the square of the constant of the Poincar\'e--Wirtinger inequality.
\end{lemma}

Notice that the function $f$ may be singular at $b$. Therefore, we need the condition $u_\Omega<b$ to ensure that the right-hand side is finite. We apply this lemma in the proof of Lemma \ref{lem.LDH1} with $f(u)=-\log(1-u)$, $u\in(0,1)$. 

\begin{proof}
The proof is based on the arguments of \cite[Lemma 4.1]{CCFG21}. We set $A:=(f(u)-f(\widehat{u}))^+$ and $A_\Omega=\m(\Omega)^{-1}\int_\Omega Adx$. Since $f(u)^2\le 2f(\widehat{u})^2 + 2[(f(u)-f(\widehat{u}))^+]^2 = 2f(\widehat{u})^2 + 2A^2$, we have
\begin{align}\label{app.fu}
  \int_\Omega f(u)^2 dx \le 4\int_\Omega(A-A_\Omega)^2 dx 
  + 4\int_\Omega A_\Omega^2 dx + 2\m(\Omega)f(\widehat{u})^{2}. 
\end{align}
We set $C_1:=2\m(\Omega)f(\widehat{u})^2$ for the last term. The first term is estimated according to the Poincar\'e--Wirtinger inequality as
\begin{align}\label{app.A-A}
  \int_\Omega(A-A_\Omega)^2 dx \le C_P\|\na A\|_{L^2(\Omega)}^2
  \le C_P\|\na f(u)\|_{L^2(\Omega)}^2.
\end{align}
Furthermore, 
\begin{align*}
  \int_\Omega(A-A_\Omega)^2 dx
  = \int_{\{A=0\}}A_\Omega^2 dx + \int_{\{A>0\}}(A-A_\Omega)^2 dx
  \ge \m(\{A=0\})A_\Omega^2. 
\end{align*}
The previous two inequalities yield
\begin{align}\label{app.AA}
  A_\Omega^2 \le \frac{C_P}{\m(\{A=0\})}\|\na f(u)\|_{L^2(\Omega)}^2.
\end{align}

We need to derive a lower bound for $\m(\{A=0\})$. To this end, we observe that $A>0$ if and only if $u>\widehat{u}$ since $f$ is assumed to be strictly increasing. Then
\begin{align*}
  \big(\m(\Omega)-\m(\{A=0\})\big)\widehat{u}
  = \m(\{A>0\})\widehat{u} = \int_{\{u>\widehat{u}\}}\widehat{u}dx 
  \leq \int_\Omega udx = \m(\Omega)u_\Omega.
\end{align*}
It follows that
\begin{align*}
  \m(\{A=0\})\ge \m(\Omega)\frac{\widehat{u}-u_\Omega}{\widehat{u}} > 0,
\end{align*}
and we infer from \eqref{app.AA} that
\begin{align*}
  A_\Omega^2 \le \frac{\widehat{u}}{\widehat{u}-u_\Omega}
  \frac{C_P}{\m(\Omega)}\|\na f(u)\|_{L^2(\Omega)}^2.
\end{align*}
Combining this inequality with \eqref{app.fu} and \eqref{app.A-A}, we obtain the result.
\end{proof}

\end{appendix}



\begin{thebibliography}{11}

\bibitem{ACFH23} D.~Abdel, C.~Chainais-Hillairet, P.~Farrell, and M.~Herda. Numerical analysis of a finite volume scheme for charge transport in perovskite solar cells. {\em IMA J. Numer. Anal.} 44 (2023), 1090--1129.

\bibitem{AGL24} D.~Abdel, A.~Glitzky, and M.~Liero. Analysis of a drift--diffusion model for perovskite solar cells. {\em Discrete Cont. Dyn. Sys.}, online first, 2024. DOI: 10.3934/dcdsb.2024081.

\bibitem{BGN22} A.~Bhattacharya, M.~Gahn, and M.~Neuss-Radu. Homogenization of a nonlinear drift--diffusion system for multiple
charged species in a porous medium. {\em Nonlin. Anal. Real World Appl.} 68 (2022), no.~103651, 28 pages.

\bibitem{Bla82} J.~Blakemore. Approximations for Fermi--Dirac integrals. {\em Solid State Electron.} 25 (1982), 1067--1076.

\bibitem{BFPR14} D.~Bothe, A.~Fischer, M.~Pierre, and G.~Rolland.
Global existence for diffusion--electromigration systems in space dimension three and higher. {\em Nonlin. Anal.} 99 (2014), 152--166.

\bibitem{CCFG21} C.~Canc\`es, C.~Chanais-Hillairet, J.~Fuhrmann, and B.~Gaudeul. A numerical-analysis-focused comparison of several finite volume schemes for a unipolar degenerate drift--diffusion model. {\em IMA J. Numer. Anal.} 41 (2021), 271--314.

\bibitem{ChLu95} Y.~Choi and R.~Lui. Multi-dimensional electrochemistry model. {\em Arch. Ration. Mech. Anal.} 130 (1995), 315--342.

\bibitem{DGJ97} P.~Degond, S.~G\'enieys, and A.~J\"ungel. A system of parabolic equations in nonequilibrium thermodynamics including thermal and electrical effects. {\em J. Math. Pures Appl.} 76 (1997), 991--1015.

\bibitem{DiRe15} K.~Disser and J.~Rehberg. Optimal Sobolev regularity for linear second-order divergence elliptic operators occuring in real-world problems. {\em SIAM J. Math. Anal.} 47 (2015), 1719--1746.


\bibitem{GaGr89} H.~Gajewski, K.~Gr\"{o}ger. Semiconductor equations for variable mobilities based on Boltzmann statistics or Fermi--Dirac statistics. {\em Math. Nachr.} 140 (1989), 7--36.

\bibitem{GaGr96} H.~Gajewski and K.~Gr\"{o}ger. Reaction--diffusion processes of electrically charged species. {\em Math. Nachr.} 177 (1996), 109--130.

\bibitem{GlHu97} A. Glitzky and R. H\"unlich. Global estimates and asymptotics for electro-reaction-diffusion systems in heterostructures. {\em Appl. Anal.} 66 (1997), 206--226.

\bibitem{GlHu05} A.~Glitzky and R.~H\"unlich. Global existence result for pair diffusion models. {\em SIAM J. Math. Anal.} 36 (2005), 1200--1225.

\bibitem{GlLi19} A.~Glitzky and M.~Liero. Instationary drift--diffusion problems with Gauss–-Fermi statistics and field-dependent mobility for organic semiconductor devices. {\em Commun. Math. Sci.} 17 (2019), 33--59.

\bibitem{GSTD13} J.~Greenlee, J.~Shank, M.~Tellekamp, and A.~Doolittle. Spatiotemporal drift--diffusion simulations of analog circuit memristors. {\em J. Appl. Phys.} 114 (2013), no.~034504, 9 pages.

\bibitem{Gro94} K.~Gr\"oger. Boundedness and continuity of solutions to linear elliptic boundary value problems in two dimensions. {\em Math. Ann.} 298 (1994), 719--728.

\bibitem{HPR19} A.~Heibig, A.~Petrov, and C.~Reichert. Solvability for a drift--diffusion system with Robin boundary conditions. {\em J. Differ. Eqs.} 267 (2019), 2331--2356.

\bibitem{IeAm20} D.~Ielmini and S.~Ambrogio. Emerging neuromorphic devices. {\em Nanotechnology} 31 (2020), no.~092001, 24 pages.

\bibitem{JJZ23} C.~Jourdana, A.~J\"ungel, and N.~Zamponi. Three-species drift--diffusion models for memristors. {\em Math. Models Meth. Appl. Sci.} 33 (2023), 2113--2156.

\bibitem{Jue96} A.~J\"ungel. Asymptotic analysis of a semiconductor model based on Fermi--Dirac statistics. {\em Math. Meth. Appl. Sci.} 19 (1996), 401--424.

\bibitem{Jue09} A.~J\"ungel. {\em Transport Equations for Semiconductors}. Lect. Notes Phys. 773, Springer, Berlin, 2009.


\bibitem{JuVe23} A.~J\"ungel and M.~Vetter. Degenerate drift--diffusion systems for memristors. Submitted for publication, 2023. arXiv:2311.16591.

\bibitem{Mla19} V.~Mladenov. {\em Advanced Memristor Modeling}. MDPI, Basel, 2019.

\bibitem{Sha68} E.~Shamir. Regularization of mixed second-order elliptic problems. {\em Israel J. Math.} 6 (1968), 150--168.

\bibitem{SBW09} D.~Strukov, J.~Borghetti, and S.~Williams.  Coupled ionic and electronic transport model of thin-film semiconductor memristive behavior. {\em Small} 5 (2009), 1058--1063.

\bibitem{TeVa20} N.~Tessler and Y.~Vaynzof. Insights from device modeling of perovskite solar cells. {\em ACS Energy Lett.} 5 (2020), 1260--1270.

\bibitem{Zei90} E.~Zeidler. {\em Nonlinear Functional Analysis and Its Applications II/A: Linear Monotone Operators}. Springer, New York, 1990. 

\end{thebibliography}
\end{document}